\documentclass[a4paper,12pt]{amsart}
\usepackage[a4paper]{geometry}
\usepackage{stmaryrd}
\usepackage{amsmath,amsthm}
\usepackage{amssymb,amsfonts}
\usepackage[hidelinks]{hyperref}
\usepackage{indentfirst}
\usepackage{tikz}
\usepackage{epigraph}
\usepackage{xfrac}
\usepackage{enumitem}   
\usepackage{graphicx}
\usepackage{calligra}
\usepackage{chngpage}
\usepackage{tikz-cd}
\usepackage{mathtools}
\DeclareSymbolFont{bbold}{U}{bbold}{m}{n}
\DeclareSymbolFontAlphabet{\mathbbm}{bbold}

\theoremstyle{plain}
	\newtheorem{theorem*}{Theorem}
	\newtheorem{thm}{Theorem}
	
	\newtheorem{coro}[thm]{Corollary}
	\newtheorem{theorem}{Theorem}[section]
\numberwithin{equation}{section}

	\newtheorem{proposition}[theorem]{Proposition}
	\newtheorem{lemma}[theorem]{Lemma}
	\newtheorem{corollary}[theorem]{Corollary}
	\newtheorem{claim}{Claim}[theorem]

\theoremstyle{definition}
	\newtheorem{definition}[theorem]{Definition}
\newtheorem{defn}[thm]{Definition}
	
	\newtheorem{example}[theorem]{Example}
	\newtheorem{conjecture}[theorem]{Conjecture}
	\newtheorem{question}[theorem]{Question}
	
\theoremstyle{remark}

\numberwithin{equation}{section}

\newcommand{\cA}{\mathcal A}
\newcommand{\bbA}{\mathbb{A}}
\newcommand{\bbB}{\mathbb{B}}
\newcommand{\bbC}{\mathbb{C}}

\newcommand{\bbL}{\mathbb{L}}
\newcommand{\bbP}{\mathbb{P}}
\newcommand{\bbQ}{\mathbb{Q}}
\newcommand{\bbT}{\mathbb{T}}

\newcommand{\cstar}{$\mathrm{C}^\ast$}
\newcommand{\awstar}{$\mathrm{AW}^\ast$}

\newcommand{\cB}{\mathcal{B}}

\newcommand{\bbR}{\mathbb{R}}
\newcommand{\bbN}{\mathbb{N}}
\newcommand{\e}{\varepsilon}
\newcommand{\cU}{\mathcal{U}}
\newcommand{\cR}{\mathcal{R}}

\newcommand{\Laffx}{\cL^{\text{aff}}_{\bar x}}

\newcommand{\fA}{\Vdash_{\bbA}}
\newcommand{\frs}{\mathfrak s}

\newcommand{\xMapsto}[2][]{\ext@arrow 0599{\Mapstofill@}{#1}{#2}}
\def\Mapstofill@{\arrowfill@{\Mapstochar\Relbar}\Relbar\Rightarrow}
\makeatother

\renewcommand{\phi}{\varphi}

\newcommand{\wstar}{$\mathrm{W}^\ast$}
\newcommand{\wst}{\mathrm{W}^\ast}

\newcommand{\cL}{\mathcal L}
\newcommand{\cLR}{\mathcal L^\mathsf{R}}

\newcommand{\bt}{\mathbf t}

\newcommand{\rAG}{\dot r_\bbA} 

\DeclareMathOperator{\cov}{cov}
\DeclareMathOperator{\Null}{Null}

\newcommand{\Boo}[1]{[\![#1]\!]} 

\title{Probably isomorphic  structures}
\date{}

\author[Ilijas Farah]{Ilijas Farah}
\address[IF]{Department of Mathematics and Statistics\\
	York University\\
	4700 Keele Street\\
	North York, Ontario\\ Canada, M3J 1P3\\
	and 
	Ma\-te\-ma\-ti\-\v cki Institut SANU\\
	Kneza Mihaila 36\\
	11\,000 Beograd, p.p. 367\\
	Serbia}
\email{ifarah@yorku.ca}
\urladdr{https://ifarah.mathstats.yorku.ca}
\thanks{I.F. is partially supported by NSERC. A.V. was funded by the Deutsche Forschungsgemeinschaft (DFG, German Research Foundation) under Germany’s Excellence Strategy – EXC 2044 – 390685587, Mathematics Münster – Dynamics – Geometry – Structure, the Deutsche Forschungsgemeinschaft (DFG, German Research Foundation) – Project-ID 427320536 – SFB 1442, and by the ERC Advanced Grant 834267 - AMAREC}

\author[Andrea Vaccaro]{Andrea Vaccaro}
\address[AV]{Institut Camille Jordan, Université Claude Bernard Lyon 1, 43 Boulevard du 11 novembre 1918, 69622, Villeurbanne, France}
	\email{vaccaro@math.univ-lyon1.fr}
\urladdr{https://sites.google.com/view/avaccaro/home}	

\subjclass[2010]{}

\begin{document}
\maketitle

\begin{abstract}
	Two structures $M, N$ in the same language are \emph{probably isomorphic} if they (or, in case of metric structures,  their completions)  are isomorphic after forcing with the Lebesgue measure algebra. We show that, if $M$ and $N$ are discrete structures, or extremal models of a non-degenerate simplicial theory, then $M$ and $N$ are probably isomorphic if and only if $L^1([0,1], M) \cong L^1([0,1], N)$. We moreover employ some of the set-theoretic arguments used to prove the aforementioned result to characterize when nontrivial ultraproducts of diffuse von Neumann algebras are tensorially prime. 
\end{abstract}

\section{Introduction}

We continue the  study of `forcing isomorphisms', analysing when nonisomorphic structures can be forced to be isomorphic by a `nice' forcing (\cite{baldwin1993forcing}, \cite{laskowski1996forcing}, \cite{shelah2025twins}).  This is just one of the unjustifiably (as will be shown below)  neglected topics initiated by Saharon and his coauthors.

Our work develops in the context of affine logic. This is a variant of continuous logic, introduced in \cite{benyaacov2024extremal}, in which only affine functions (instead of all continuous functions, as in continuous logic) are allowed as connectives. This provides a setting where models of an affine theory are closed under taking convex combinations and even direct integrals, which in this framework have a natural formulation (see \cite[\S 8]{benyaacov2024extremal}) generalizing several classical constructions from functional analysis, such as direct integrals of Hilbert spaces or of von Neumann algebras (see e.g. \cite[Chapter IV]{Tak:TheoryI}). 

The type spaces of an affine theory can be naturally endowed with the structure of compact convex sets, and the  \emph{extremal models} of a theory are those that correspond to its extremal types (see \S\ref{S.Affine} for details). These structures are particularly important for \emph{simplicial} theories, namely affine theories whose type spaces are Choquet simplices: in this case every model of the theory can be decomposed as direct integral of extremal models---a specific form of this result had been in use for decades in the framework of separably representable von Neumann algebras (see e.g. \cite[\S IV.8]{Tak:TheoryI}), but its extension to the nonseparable setting is among the numerous novelties presented in \cite{benyaacov2024extremal}.

Let $M$ be a metric structure  and $(\Omega,  \Sigma, \mu)$ be  a probability space.
In this paper we explore a connection between the study of structures of the form $L^1(\Omega, M)$---which correspond precisely to direct integrals with fibre constantly equal to $M$ (see Definition \ref{Def.L1})---and forcing by measure algebras.

\begin{thm} \label{T.isomorphic_randomizations} Let $(\Omega, \Sigma, \mu)$ be a probability measure space and let $\bbA$ be the measure algebra obtained as the quotient of $\Sigma$ by the null ideal. Suppose that $M$ and $N$ are single-sorted discrete structures, or extremal models of the same non-degenerate, simplicial theory. Then $L^1(\Omega, M) \cong L^1(\Omega,N)$ if and only if $\bbA$ forces that the completions of $M$ and $N$ are isomorphic, or in the discrete case that $M$ and $N$ are isomorphic. 
\end{thm}

The proof of this result relies on an analysis of names of elements in $M$---or rather, of its completion in the generic extension---reminiscent of \cite{jech1985abstract, vaccaro2017generic}, and combines it with the theory of randomizations as defined by Ben Yaacov and Kiesler (randomizations of discrete structures were first introduced in \cite{keisler1999randomizing, benyaacov2009randomizations} and the definition was later extended to metric structures in \cite{benyaacov2013theories}). More precisely, we show that  the set of names of elements in a structure $M$ is isomorphic to the randomization of $M$ (see Proposition \ref{prop:random}).

As an application, we answer \cite[Question~12.9]{benyaacov2024extremal}, which asks whether there are a simplicial theory~$T$ and non-isomorphic extremal models $M$ and $N$ of $T$ such that $L^1(\Omega,M)$ and $L^1(\Omega,N)$ are isomorphic. 
Since a complete, single-sorted, discrete theory from classical first order logic (when suitably interpreted as an affine theory) is simplicial and its discrete models are extremal (\cite[\S 26]{benyaacov2024extremal}), Corollary~\ref{C1} answers this question in the affirmative. 

\begin{coro}
\label{C1}
	Consider the non-isomorphic linear orders $M \coloneqq  (\bbR,<)$ and  $N \coloneqq (\bbR\setminus \{0\}, <)$ as metric structures with the discrete metric. Then $L^1([0,1], M) \cong L^1([0,1], N)$, where $[0,1]$ is interpreted as a probability measure space with the Le\-besgue measure.
\end{coro}

The second part of the paper focuses instead on tensorial primeness of tracial von Neumann algebras.
A \cstar-algebra is \emph{tensorially prime} if it is not a tensor product of two infinite-dimensional \cstar-algebras. A  nontrivial ultrapower of a \cstar-algebra is tensorially prime and the same holds for coronas of $\sigma$-unital \cstar-algebras and all other countably degree-1 saturated \cstar-algebras (\cite{ghasemi2015isomorphisms}, see also \cite[\S 15.4]{Fa:STCstar}). 
 In case of tracial von Neumann algebras, we will argue that the following is the correct definition. 

\begin{defn}\label{Def.prime}
 A tracial von Neumann algebra $(M, \tau)$ is \emph{tensorially indecomposable} (or \emph{tensorially prime}) if it is not isomorphic to a tensor product of two von Neumann algebras whose unit balls are not compact in the 2-norm induced by $\tau$.
\end{defn}

In case of \cstar-algebras, having a unit ball which is not compact in the operator topology is equivalent to being infinite-dimensional. However,  
by Example~\ref{Ex.Prime} below,  if $M$ is a type II$_1$ tracial von Neumann algebra with diffuse center, then its ultrapower $M^{\cU}$ absorbs $\ell^\infty(\bbN)$ tensorially (our proof uses the Continuum Hypothesis, but we are under the impression, with no evidence to support it,  that this assumption is unnecessary).

Ultrapowers of II$_1$ factors, as well as the relative commutant of the hyperfinite factor $R$ in its ultrapower, $R^{\cU}\cap R'$,  are tensorially prime. 
 This result was first published in \cite[Theorem~4.5]{fang2006central}, using a key result coming from \cite{popa1983orthogonal}.
 A condition sufficient for tensorial primeness of a II$_1$ factor $M$ was isolated in \cite{popa2022topics} and   \cite[Theorem~1.4]{hiatt2024singular}---namely that any two copies of $L^\infty[0,1]$ are unitarily equivalent---and another sufficient condition was isolated in  \cite[Proposition~3.16]{gao2024conjugacy}---that any two Haar unitaries are unitarily equivalent.   The latter condition easily implies the former, and it is not known whether the two conditions are equivalent (see \cite[Remark~1.8]{hiatt2024singular}). 
Tensorial primeness of ultrapowers of II$_1$ factors begs the question which ultrapowers of diffuse tracial von Neumann algebras and relative commutants in such ultrapowers are tensorially prime. The following is proved after Theorem~\ref{T.CountablySaturated}.

\begin{thm}\label{T.UltrapowerTensoriallyIndecomposable}
	Suppose that $M$ is an ultraproduct of tracial von Neumann algebras associated with a nonprincipal ultrafilter on $\bbN$ . If $M$ is not type I, then it is tensorially prime. Moreover, if $N$ is a \wstar-subalgebra of $M$ with separable predual, and $M\cap N'$  is not type I, then it is tensorially prime. 
	
	On the other hand, if $N$ is a tracial, diffuse von Neumann algebra of type I, and~$\cU$ is a nonprincipal ultrafilter over $\bbN$, then $N^\cU$ is not tensorially prime.
\end{thm}


\begin{coro}\label{T.ultrapower}
	Suppose that $M$ is a II$_1$ factor and $\cU$ is a nonprincipal ultrafilter over~$\bbN$. Then $(L^\infty[0,1]\bar\otimes M)^\cU$ is not isomorphic to $L^\infty(\Omega, \mu )\bar \otimes N$ for any type II$_1$ tracial von Neumann algebra $N$ and any probability measure space $(\Omega, \mu)$.
\end{coro}
		
		Our proofs use set-theoretic forcing, and in fact they rely on the same analysis of names that yields Theorem \ref{T.isomorphic_randomizations}. Immediately upon seeing the first draft of our paper, Adrian Ioana provided a proof of Theorem~\ref{T.UltrapowerTensoriallyIndecomposable} (and therefore of Corollary~\ref{T.ultrapower}) that uses only the standard techniques from Popa's deformation and rigidity theory. This proof is included in the final appendix with his kind permission. 
		
		On the other hand, the proof of Corollary~\ref{C1} via Theorem~\ref{T.isomorphic_randomizations} and forcing gives, in our opinion, the most natural and most enlightening approach to this problem. 
		
		\subsection*{Summary of the paper} Section \ref{S.preliminaries} focuses on preliminaries. In Section \ref{S.names} we study the class $M^\bbA$ of $\bbA$-names, for $\bbA$ a measure algebra, for (the completion of) a metric structure $M$ in a language $\cL$. We show that $M^\bbA$, or rather a natural quotient of it, can be interpreted as a structure  isomorphic to $L^1(\Omega, M)$---with $(\Omega ,\mu)$ being the probability space corresponding to $\bbA$---when both are interpreted in the randomization language $\cLR$ (whose definition is recalled in \S\ref{S.random}). This correspondence is essential to prove Theorem \ref{T.isomorphic_randomizations}, whose proof is given in Section \ref{S.X->L1}. Section \ref{S.vNa} contains the proof of Theorem \ref{T.UltrapowerTensoriallyIndecomposable}, and  Section \ref{S.Concluding} is devoted to final remarks. Appendix \S\ref{S.def/rg} contains Ioana's proof of tensorial primeness of ultrapowers of type II$_1$ von Neumann algebras. Readers interested only in von Neumann algebras can go straight to Section~\ref{S.vNa}, skip its parts that use forcing, and move on to Appendix \ref{S.def/rg}. 
		
	\subsection*{Acknowledgments}  We thank Chris Laskowski and Menachem Magidor for turning our attention to \cite[Example 5.1]{laskowski1996forcing} on which the proof of Corollary~\ref{C1} is based. 
	IF would also like to thank the attentive audience of the seminar given in May 2025 at KU Leuven, and Stefaan Vaes in particular, at which the results of this paper were first presented. AV would finally like to thank Soham Chakraborty and Srivatsav Kunnawalkam Elayavalli for their prompt and useful insights.
Finally, we would like to thank Tom\' as Ibarlucia, Adrian Ioana, David Jekel, Jennifer Pi, Sorin Popa for helpful remarks on the first draft of this paper and to the anonymous referee for a helpful report. Following the referee's advice, we used ChatGPT for proofreading revised version of the paper.

\section{Preliminaries} \label{S.preliminaries}

\subsection{Measure algebras} \label{S.MeasureAlgebras} A \emph{measure algebra} consists of a $\sigma$-complete Boolean algebra $B$ paired with a strictly positive probability measure $\mu$, namely a function $\mu \colon B \to \bbR$ such that
\begin{enumerate}
\item $\mu(0_B) = 0$,
\item $\mu(a) > 0$ for all $a \in B \setminus \{ 0 \}$,
\item $\mu \left(\sum_{n\in \bbN} a_n\right) = \sum_{n\in \bbN} \mu (a_n)$ for all sequences of pairwise disjoint elements $(a_n)_{n\in \bbN}$ in $B$,
\item $\mu(1_B) = 1$.
\end{enumerate}

Given a probability measure space $(\Omega, \Sigma, \mu)$, which we often denote $(\Omega, \mu)$ omitting the $\sigma$-algebra $\Sigma$, its measure algebra $\bbA$ is the quotient of $\Sigma$ by the null ideal $\{ A \in \Sigma : \mu (A) = 0 \}$, paired with measure induced by $\mu$ (which, with a slight abuse of notation, we still denote $\mu$). Note that every measure algebra is the measure algebra of some probability measure space (see e.g. \cite[\S 321J]{Fr:MT3}).

A central fact in the theory of measure algebras is Maharam's theorem (\cite{maharam},  see also \cite[\S 331]{Fr:MT3}), stating that the isomorphism class of a homogeneous infinite measure algebra is determined by its weight.
We briefly recall this fundamental result, used repeatedly in the following sections (see \cite[Chapters 31, 32]{Fr:MT3} for a thorough introduction). 

Given a measure algebra $\bbA$, its \emph{weight} is the least size of a subset $B \subseteq \bbA$ such that the $\sigma$-algebra generated by $B$ is $\bbA$ itself. This coincides with the density character of $\bbA$ when considered with respect to the natural metric $d(a,b)\coloneqq\mu(a\Delta b)$. For a measure algebra $\bbA$ and a positive set $a\in \bbA$ we can consider the measure algebra $\bbA_a\coloneqq \{b\in \bbA\mid b\leq a\}$ with respect to the normalized measure $\mu_a(b)=\mu(b)/\mu(a)$.  A measure algebra $\bbA$ is \emph{(Maharam) homogeneous} if the weight of $\bbA_a$ is equal to the weight of $\bbA$ for all positive $a\in \bbA$.

Maharam's theorem asserts that, for every infinite cardinal $\kappa$, the measure algebra of $\{0,1\}^\kappa$, equipped with the product measure $\mu_\kappa$, is the unique, up to measure-preserving isomorphism,  homogeneous algebra of weight~$\kappa$. Since an elementary argument shows that every measure algebra is a direct sum of at most countably many homogeneous  measure algebras, Maharam's theorem yields a complete classification of diffuse (probability) measure algebras as a corollary. More precisely, for every $\bbA$ there is a countable set $\mathcal{I}$, some reals $0<r_i\leq 1$ such that $\sum_{i \in \mathcal{I}} r_i=1$, infinite cardinals~$(\kappa_i)_{i \in \mathcal{I}}$, and a family $(a_i)_{i \in \mathcal{I}}$ of pairwise disjoint elements in $\bbA$ whose union has full measure, such that $\mu(a_i)=r_i$, $\bbA_{a_i}$ is homogeneous of weight $\kappa_i$ and $\bbA = \bigoplus_{i \in \mathcal{I}} \bbA_{a_i}$.
This immediately implies the following. 

\begin{proposition}\label{P.Abelian-tvNa}
Let $A$ be an abelian von Neumann algebra. There is a sequence of cardinals $(\kappa_n)_{n \in \bbN}$ and a sequence of pairwise orthogonal projections $(p_n)_{n \in \bbN}$ in~$A$ such that $A = \bigoplus_{n \in \bbN} p_n A p_n$ and such that each corner $p_n A p_n$ is isomorphic to $L^\infty( \{ 0,1 \}^{\kappa_n}, \mu_{\kappa_n})$, where $(\{ 0,1 \}^{\kappa_n}, \mu_{\kappa_n})$ is the homogeneous measure algebra on $\{0,1\}^{\kappa_n}$ with the product measure. The sequence $(\kappa_n)_{n \in \bbN}$ can be eventually zero.
\end{proposition}

\subsection{Metric structures and $L^1$-structures} 

Languages and metric structures are considered in the sense of \cite{benyaacov2008model} and \cite{benyaacov2024extremal}, and we refer to these works for the relevant background. Those references mainly focus on single-sorted structures, but the theory can be naturally adapted to multi-sorted structures (this is generally convenient when dealing with unbounded metric spaces such as $C^\ast$-algebras and tracial von Neumann algebras; see \cite{hart:continuous_mt} for a thorough introduction).

To streamline the exposition, we restrict ourselves to single-sorted metric structures, with the proviso that everything we say also holds in the multi-sorted case. We shall also sometimes consider theories from classical first order logic and models of these theories. These discrete theories naturally give rise to theories in continuous logic (see the beginning of  \cite[\S 26]{benyaacov2024extremal}) which are called \emph{classical} theories, where all basic predicates are $\{0,1\}$-valued. Models of discrete theories then naturally become models of classical theories, after interpreting them---as we shall always do---as metric structures with the discrete metric.

The domain of a structure is always assumed to be $(M, d^M)$, where $M$ is a non-empty, bounded, complete metric space with distance $d^M$.
This restriction does not pose serious limitations when working with tracial von Neumann algebras $(M, \tau)$, which will always be identified with their closed unit balls and interpreted as metric structures $(M, \| \cdot \|_{1,\tau})$\footnote{Given a von Neumann algebra $M$ and a faithful trace $\tau$ on $M$, the 1-norm is defined by $\| a \|_{1, \tau} \coloneqq \tau((a^\ast a)^{1/2})$, for $a \in M$. While the standard axiomatization of tracial von Neumann algebras uses the $\|\cdot\|_{2,\tau}$-norm, the two norms are equivalent on operator norm-bounded  balls by \cite[Lemma~29.1]{benyaacov2024extremal}. The theory with respect to 1-norm fits, however, within the affine framework of \cite{benyaacov2024extremal}, making the methods therein applicable. The difference is of only syntactic nature.} in the language $\cL_{\text{TvN}}$ as in \cite[\S 29]{benyaacov2024extremal} (see also \cite[\S 2.3.2]{FaHaSh:Model2} for an equivalent presentation of tracial von Neumann algebras via multiple domains and with respect to the standard 2-norm $\|\cdot\|_{2,\tau}$). We moreover assume, again without loss of generality, that all predicates and formulas have range contained in the interval $[0,1]$ (in the case of von Neumann algebras this can be done by shifting and rescaling).

In this paper, we focus on the study of certain direct integrals $\int^\oplus_\Omega M_\omega \, d\mu(\omega)$ of metric structures as defined in \cite[\S 8]{benyaacov2024extremal},
to which we refer for the precise definitions and a thorough presentation of this construction. 
In case all $M_\omega$ are tracial von Neumann algebras with separable predual, this definition agrees with the standard definition as in \cite[Chapter IV]{Tak:TheoryI}. 

In this paper we exclusively deal with structures of the form $L^1(\Omega, \mu, M)$, abbreviated $L^1( \Omega, M)$, which correspond to integrals of a constant measurable field with $M_\omega = M$ for all $\omega \in \Omega$, and with \emph{pointwise enumeration} (in the sense of \cite[Definition 8.1.iii]{benyaacov2024extremal}) given by constant functions with values in a dense subset of $M$. We briefly recall their definition.

\begin{definition}[{\cite[Definition 8.15]{benyaacov2024extremal}}]\label{Def.L1}
Let $(\Omega, \mu)$ be a probability measure space. Given a metric structure $(M, d^M)$ in some language $\cL$, we can define a metric structure $(L^1(\Omega,M), d^{L^1(\Omega,M)})$ in $\cL$ as follows.
\begin{enumerate}
\item \emph{Simple sections} are functions $s \colon \Omega \to M$ such that there is a finite partition $A_0, \dots, A_{n-1}$ of $\Omega$ and $a_0, \dots, a_{n-1} \in M$ such that $s = \sum_{i < n} a_i \mathbf{1}_{A_i}$, where $a_i \mathbf{1}_{A_i} \colon A_i \to M$ is the constant function with value $a_i$. A \emph{measurable section} is a function $f \colon \Omega \to M$ that is the pointwise limit of simple sections for $\mu$-almost every $\omega\in\Omega$ (\cite[Lemma 8.5]{benyaacov2024extremal}).

\item Given two measurable sections $f, g \colon \Omega \to M$, the function from $\Omega$ into $\bbR$ that maps $\omega \mapsto d^M(f(\omega), g(\omega))$ is $\mu$-measurable (\cite[Lemma 8.4]{benyaacov2024extremal}). It is then possible to define a pseudometric on the set of all measurable sections as follows
\[
d^{L^1(\Omega, M)}(f,g) \coloneqq \int_\Omega d^M(f(\omega), g(\omega))\, d\mu(\omega).
\]

\item The set $L^1(\Omega,M)$ is the quotient of the set of all measurable sections, where two sections are identified if they are equal for $\mu$-almost every $\omega\in\Omega$. Then $(L^1(\Omega,M), d^{L^1(\Omega,M)})$ is a complete metric space in which the collection (of equivalence classes) of simple sections is dense (\cite[Proposition 8.8]{benyaacov2024extremal}).

\item \label{4.D.L.1}Given a $k$-ary predicate symbol $R$ in $\cL$ and a $k$-tuple $\bar f \coloneqq (f_0, \dots, f_{k-1})$ in $L^1(\Omega, M)$, the function $\omega \mapsto R^M(\bar f(\omega))$, where $\bar f(\omega)$ is an abbreviation for the $k$-tuple $(f_0(\omega), \dots, f_{k-1}(\omega))$, is $\mu$-measurable (\cite[Lemma 8.7]{benyaacov2024extremal}). We can then define the interpretation of $R$ as follows
\[
R^{L^1(\Omega, M)}(\bar f) \coloneqq \int_\Omega R^M(\bar f(\omega))\, d\mu(\omega).
\]
The interpretation of function and constant symbols in $\cL$ is defined similarly.
\end{enumerate}
\end{definition}

Note that, if $M$ is a tracial von Neumann algebra identified with its unit ball and interpreted as a metric structure in $\cL_{\text{TvN}}$, then the $\cL_{\text{TvN}}$-structure $L^1(\Omega, \mu, M)$ corresponds to (the closed unit ball of) $L^\infty(\Omega, \mu) \bar \otimes M$.

\subsection{Affine logic}\label{S.Affine}
The framework developed in \cite{benyaacov2024extremal} is that of \emph{affine} (continuous) logic: the definition of formulas is similar to that in continuous logic (as in, for instance, \cite{benyaacov2008model}) with the additional restriction that connectives are constrained to be functions $[0,1]^n \to [0,1]$ of the form $\bar x \mapsto f(\bar x) + b$ where $f$ is linear and $b \in \bbR$, and the moduli of continuity for symbols in the language are always required to be continuous and convex functions. All the usual model theoretic terminology (e.g. elementary equivalence, definable set, type, theory, etc.) has a natural interpretation in the affine setting. We recall only the notions that are strictly necessary for this paper, and refer the reader to \cite{benyaacov2024extremal} for all the missing details.

Fix a continuous language $\cL$ and denote by $\cL_{\bar x}$  the set of all formulas in the sense of continuous model theory (see e.g. \cite[\S 3]{benyaacov2008model}) and by $\Laffx$ the set of all affine formulas (as in \cite[\S 2]{benyaacov2024extremal})  in the language $\cL$ whose free variables are among those in tuple $\bar x$. Given a metric structure $M$ and $\phi(\bar x) \in \cL_{\bar x}$, we denote by $\phi^M \colon M^{| \bar x |} \to \bbR$ the interpretation of $\phi$ on $M$, where $| \bar x |$ is the length of the tuple $\bar x$. 

We say that a substructre $N \subseteq M$ is an \emph{affine substructure}  of $M$, $N \preceq^{\text{aff}} M$ in symbols, if $\phi^M(\bar a) = \phi^N(\bar a)$ for all $\phi(\bar x) \in \Laffx$ and $\bar a \in N^{| \bar x |}$. Similarly, $N$ is an \emph{elementary substructure} of $M$, which we abbreviate as $N \preceq M$, if the previous condition holds for all $\phi(\bar x) \in \cL_{\bar x}$. An isometric embedding $\Phi \colon N \to M$ between metric structures is \emph{affine} if $\Phi(N) \preceq^{\text{aff}} M$, and it is \emph{elementary} if $\Phi(N) \preceq M$. 

As in \cite[\S 3]{benyaacov2024extremal}, the set $\Laffx$ can be naturally given a structure of order unit space, and an \emph{affine type} in the variables $\bar x$ is a state on $\Laffx$. The space of all affine types in $\bar x$ is denoted $S_{\bar x}^{\text{aff}}(\cL)$, and it is a compact convex set. Given an affine theory $T$, one can define the equivalence relation $\equiv_T$ on $\Laffx(\cL)$ by setting $\phi(\bar x) \equiv_T \psi(\bar x)$ if and only if $\phi^M(\bar a) = \psi^M(\bar a)$ for every model $M$ of $T$ and all tuples $\bar a \in M^{| \bar x |}$. A \emph{type on $T$} is a state on $\Laffx (T) \coloneqq \Laffx/ \equiv_T$ and the set of types in $T$, denoted $S_{\bar x}^{\text{aff}}(T)$, is again a  compact convex set.

A theory $T$ is called \emph{simplicial}
if $S_{\bar x}^{\text{aff}}(T)$ is a Choquet simplex for every tuple of variables $\bar x$ (\cite[Definition 11.1]{benyaacov2024extremal}) and it is \emph{non-degenerate} if all of its models have at least two elements (in at least one sort, in the multi-sorted case; see \cite[Definition 8.23]{benyaacov2024extremal}). An example of a simplicial, non-degenerate theory is that of tracial von Neumann algebras, as described in \cite[\S 29]{benyaacov2024extremal}.

An affine type is \emph{extreme} if it is an extreme point of $S_{\bar x}^{\text{aff}}(\cL)$. Given an affine theory $T$, a metric structure $M$ is an \emph{extremal model} of $T$ if tuples in $M$ only realize extreme types. Looking again at the theory of tracial von Neumann algebras, extremal models correspond to tracial factors (see \cite[\S 29]{benyaacov2024extremal}). Recall furthermore that discrete models of a classical theory are extremal and that, if complete, a single-sorted classical theory is simplicial (\cite[\S 26]{benyaacov2024extremal}).

Given a metric structure $M$ and a probability space $(\Omega, \mu)$ one can define a function $\llbracket \cdot\rrbracket $ from the space of formulas, possibly with parameters in $L^1(\Omega, M)$, into the space of random variables $L^1(\Omega, [0,1])$ defined by
\begin{equation} \label{eq:boo}
\llbracket\varphi(\bar a)\rrbracket(\omega)\coloneqq\varphi^{M}(\bar a(\omega)), \quad \omega \in \Omega, \, \phi \in \cL_{\bar x}, \, \bar a \in L^1(\Omega, M)^{| \bar x |}.
\end{equation}
Indeed,  the map $\omega \mapsto \phi(\bar a(\omega))$ is $\mu$-measurable and $\llbracket\varphi(\bar a)\rrbracket \in L^1(\Omega, [0,1])$. The following analog of \L o\'s's theorem for affine logic, which we state only for $L^1$-structures, has been proved in \cite[Theorem 8.11]{benyaacov2024extremal} in the more general setting of direct integrals of metric structures.

\begin{theorem}[{\cite[Theorem 8.11]{benyaacov2024extremal}}] \label{thm:Los}
Let $M$ be a metric structure in some language $\cL$ and let $(\Omega,\mu)$ be a probability space. For every $\phi(\bar x) \in \Laffx$ and $ \bar f \in L^1(\Omega, M)^{| \bar x |}$, the function $\omega \mapsto \phi^{M}(\bar f(\omega))$ is $\mu$-measurable and
\[
\phi^{L^1(\Omega, M)}(\bar f) = \int_\Omega \phi^{M}(\bar f(\omega)) \, d \mu(\omega).
\]
\end{theorem}

One of the most relevant (to our purposes) results proved in \cite{benyaacov2024extremal} is that models of non-degenerate simplicial theories admit a unique integral decomposition (up to an appropriate notion of isomorphism of measurable fields of structures). Our proof of Theorem \ref{T.isomorphic_randomizations} makes essential use of the following uniqueness theorem for decomposition of $L^1$-structures, which in its original form in \cite[Theorem~12.16]{benyaacov2024extremal} is stated for direct integrals.

\begin{theorem}[{\cite[Theorem~12.16]{benyaacov2024extremal}}] \label{L.DetectingCoordinates} Let  $(\Omega, \mu)$ and $(\Xi, \nu)$ be probability measure spaces, and suppose that $M$, $N$ are extremal models of the same non-degenerate, simplicial, affine theory.
		\begin{enumerate}
		\item 	 If $\Phi\colon L^1(\Omega, M) \to L^1(\Xi, N)$ is an affine embedding then there is an isometric linear map\footnote{In \cite[Definition~12.9]{benyaacov2024extremal} $\frs$ is an embedding between measure algebras, but we identify it with its canonical extension to the map between the associated $L^1$-spaces.} 
		\[
		\frs_{\Phi}\colon L^1(\Omega,[0,1])\to L^1(\Xi,[0,1])
		\] 
		such that  for every $\phi(\bar x) \in \Laffx$ and every tuple of variables $\bar a \in L^1(\Omega,M)^{| \bar x |}$ we have 
		\[
		\frs_{\Phi}\llbracket\phi(\bar a)\rrbracket=\llbracket \phi(\Phi(\bar a))\rrbracket. 
		\]
		\item 	If $\Phi$  is moreover an isomorphism then $\frs_{\Phi}$ is an isomorphism. 
			\end{enumerate}
 \end{theorem}

We remark that Theorem~\ref{L.DetectingCoordinates} is considerably easier to prove for the theory of tracial von Neumann algebras because in that case the measure algebras associated to $(\Omega, \mu)$ and $(\Xi, \nu)$ are definable, being the lattices of projections of the centers of the von Neumann algebras in question.

\subsection{Randomizations} \label{S.random}
We recall the notion of {randomization} of a continuous language as in \cite{benyaacov2013theories} (to which we refer for all the missing details; see also \cite{keisler1999randomizing, benyaacov2009randomizations} for analogous definitions for discrete languages). Given a continuous language~$\cL$, the \emph{randomization language} $\cLR$ contains all sorts of $\cL$ (called main sorts) as well as an auxiliary sort. Every function symbol $F$ of $\cL$ is retained in $\cLR$, with the same sort and modulus of uniform continuity. 
For every predicate symbol $R$ of $\cL$ (including the distance function), $\cLR$ contains a function symbol $\Boo R$ from the corresponding main sorts to the auxiliary sort. It retains the same modulus of continuity as $R$. 
The auxiliary sort is equipped with the structure of the space of \emph{random variables} as in  \cite[\S 2]{benyaacov2013theories}. 

More concretely, given a metric structure $M$ and a probability measure space $(\Omega, \mu)$, the \emph{randomization} of $M$ is an $\cLR$-structure obtained from $(L^1(\Omega,M), d^{L^1(\Omega,M)})$ as in Definition~\ref{Def.L1} by adding the auxiliary sort $L^1(\Omega, [0,1])$ and  replacing~\eqref{4.D.L.1}   with the following variant.
\begin{enumerate}[label=(\arabic*$^{\mathsf{R}}$)] \setcounter{enumi}{3}
\item Given a $k$-ary predicate symbol $R$ in $\cL$ and a $k$-tuple $\bar f \coloneqq (f_0, \dots, f_{k-1})$ in $L^1(\Omega, M)$, interpret $\Boo R (\bar f)$ as the equivalence class of the measurable function $\omega \mapsto R^M(\bar f(\omega))$, where $\bar f(\omega)$ is an abbreviation for the $k$-tuple $(f_0(\omega), \dots, f_{k-1}(\omega))$. 
\end{enumerate}

Letting $\bbA$ be the measure algebra associated to $(\Omega, \mu)$, following \cite{benyaacov2013theories}, we use the notation $(\widehat M, \widehat \bbA)$ to denote the {randomization} in $\cLR$ of $M$ just defined, in order to distinguish it from the metric $\cL$-structure  $(L^1(\Omega,M), d^{L^1(\Omega,M)})$ as in Definition~\ref{Def.L1}.

This describes  the minimalist version of $\cLR$, as introduced in \cite[\S 3]{benyaacov2013theories}. In the original definitions from \cite{keisler1999randomizing, benyaacov2009randomizations}, the randomization language $\cLR$ contains a function symbol $\Boo\varphi$, with image in the auxiliary sort, for every $\cL$-formula $\phi$, interpreted as in \eqref{eq:boo} in the randomization of $M$. Note, however, that by \cite[Theorem~3.14]{benyaacov2013theories} these additional symbols are definable for randomizations of metric structures in the minimalist $\cLR$.

\subsection{Countable quantifier-free saturation}\label{S.cqf-saturated}

We recall the definition of countable quantifier-free saturation only for tracial von Neumann algebras. A tracial von Neumann algebra $(M,\tau)$ is \emph{countably quantifier-free saturated} if whenever $p_n(\bar x)$ are *-polynomials in non-commuting variables with coefficients in $M$ and $r_n\geq 0$, for $n\in \bbN$, the system of equations (usually called \emph{conditions}) $\|p_n(\bar x)\|_{2, \tau}=r_n$ has an exact solution in  the operator unit ball of $M$ if and only if, for every $n\geq 1$, there is a tuple $\bar a_n$ of elements in the operator unit ball of $M$ such that
\[
\max_{j<n}|\|p_j(\bar a_n)\|_{2,\tau}-r_j|_2<1/n.
\] 
Such system of equations is usually called a \emph{(partial) type},\footnote{All types in this paper will be partial.} and countable quantifier-free saturation of $M$ asserts that a type is realized  in $M$ if and only if it is approximately satisfiable in $M$. 

If $M=\prod^{\cU} M_n$ is an ultraproduct of tracial von Neumann algebras  associated with a nonprincipal ultrafilter on $\bbN$ and  $N$ is a \wstar-subalgebra of $M$ with separable predual then both $M$ and $M\cap N'$ are countably quantifier-free saturated (\cite[Theorem 16.4.1 and Corollary 16.7.2]{Fa:STCstar}). While $M$ is even fully saturated, we do not know whether this is the case with $M\cap N'$. For comparison, the relative commutant of a separable \cstar-subalgebra of a $C^\ast$-ultrapower is countably quantifier-free saturated, but not necessarily fully saturated (\cite[Proposition~6.3]{farah2023obstructions}). 

\subsection{Forcing with measure algebras}
We refer to \cite{kunen2014set, Jech:ST} for an introduction to the forcing method and for the use of measure algebras as forcing notions (see  also \cite{solovay1970model}). 

If $\bbA$ is a measure algebra, then a \emph{maximal antichain} in $\bbA$ is a family $\cA$ of pairwise disjoint elements of $\bbA$ such that $\bigcup \cA$ has full measure. An important property of $\bbA$ as a forcing notion is that every maximal antichain is countable, a property that had been misnomered as \emph{countable chain condition}, or \emph{c.c.c.}. 

A \emph{generic filter} is a filter on $\bbA$ that intersects every maximal antichain. 
If $\bbA$ is a measure algebra, then by Maharam's theorem there is a maximal antichain $\cA$ in $\bbA$ such that for each $a\in \cA$ the measure algebra $\bbA_a$ is homogeneous. Therefore for every measure algebra $\bbA$, every generic filter corresponds to the generic filter in some homogeneous algebra of the form $\bbA_a$. 

If $\bbA$ is an atomless measure algebra then a generic filter does not exist. However, a standard combination of the L\" owenheim--Skolem Theorem and of the Mostowski collapse (see e.g., \cite[IV.1, IV.2]{kunen2014set}) shows that assuming the existence of a generic filter $G$ in a larger universe, called \emph{forcing extension} and denoted $V[G]$ ($V$ is the standard notation for von Neumann's universe) serves as a valid proof technique. 
We may therefore  assume that  every condition $p$ in $\bbA$  is included in  a generic filter~$G$ that exists in an extension of $V$. 

Every element of $V[G]$ has a \emph{name} in $V$: the set of all $\bbA$-names (\cite[Definition IV.2.5]{kunen2014set}) is denoted as $V^\bbA$. We let $\dot G$ be the canonical name for a generic filter $G$. By Maharam's theorem (\S\ref{S.MeasureAlgebras}) we may assume that the measure space associated to~$\bbA$ is of the form $(\{0,1\}^\kappa,\mu_\kappa)$. For every $\xi<\kappa$ it is forced that exactly one of the sets $\{\alpha\in \{0,1\}^\kappa\mid \alpha(\xi)=j\}$ for $j=0,1$ belongs to $\dot G$. Therefore $\dot G$ determines a generic element of $\{0,1\}^\kappa$, whose canonical name is denoted $\rAG$.  It is suggestively called `random real' in case when $\bbA$ is the Lebesgue measure algebra.

The properties of $V[G]$ are described by the  \emph{forcing language}, an expansion of the language of ZFC, by names for elements of $V[G]$. If $G$ is a generic filter, then every name $\dot a$ has an evaluation, which we denote $\dot a_G$ (\cite[Definition IV.2.7]{kunen2014set}), and $V[G]$ is the class of all such evaluations. In other words, $V[G]$ is obtained by evaluating names by $G$. An argument using coding for Baire-measurable subsets of $\{0,1\}^\kappa$ shows that $(\dot r_\bbA)_G$ uniquely determines~$G$, but we will not need this fact. 

 If $\varphi$ is a sentence of the forcing language and $p\in \bbA$, we say that $p$ \emph{forces} $\varphi$, and write $p\fA \varphi$, if the sentence obtained from $\varphi$, by evaluating all names that occur in it, holds in $V[G]$ for every generic $G$ with $p \in G$. We write $\fA \phi$ to abbreviate $1_\bbA \fA \phi$.
 	
	 In the following proposition we collect all the key facts about forcing that we will need. Some of them apply to an arbitrary forcing notion, while some are specific to measure algebras. 

\begin{proposition}\label{P.forcing}  Suppose that $\bbA$ is a measure algebra with measure $\mu$ and $G\subseteq \bbA$ is a generic filter. Let $\phi, \psi$ be sentences in the forcing language.
	\begin{enumerate}
		\item \label{4.0.forcing} No $p \in \bbA$ can force both $\phi$ and $\neg \phi$.
		\item 
		If $q\le p\in \bbA$ and $p\fA\varphi$, then $q \fA \varphi$. 
		\item \label{4.1.forcing} 
		 $p\fA\varphi$ if and only if no $q\leq p$ forces $\lnot \varphi$. 
		\item \label{0.forcing} If $p\fA (\exists x)\psi(x)$ then there is a name $\dot a$ in $V^\bbA$ such that $p\fA\psi(\dot a)$. 
		


		\item \label{5.forcing} Suppose that $(q_n)_{n\in \bbN}$ are elements of $\bbA$  such that for every pair of disjoint finite subsets $X$ and $Y$ of $\bbN$ we have (with the convention that $\bigcap_{n\in \emptyset} q_n=1_\bbA$)
		\[
		\textstyle\mu\left(\bigcap_{n\in X}q_n\cap \bigcap_{n\in Y} (1_\bbA - q_n)\right)=2^{-|X|-|Y|}. 
		\]
			Then the set $Z_G=\{n\mid q_n\in G\}$ belongs to $V[G]$, but not to $V$. 
	\end{enumerate}
	
\end{proposition}
\begin{proof}
		\eqref{4.0.forcing}--\eqref{4.1.forcing} . See \cite[Lemma IV.2.30]{kunen2014set}.
			
	\eqref{0.forcing}. See \cite[Lemma IV.2.32]{kunen2014set}.
	


	
	\eqref{5.forcing}. Let $(q_n)_{n \in \bbN}$ be as in the statement and fix some $Z\subseteq \bbN$. For $m\in \bbN$, interpreting $m$ as the set $\{ 0, \dots, m-1 \}$, define 
	\[
	p_m \coloneqq \bigcap_{n\in Z\cap m}q_n\cap \bigcap_{n\in m\setminus Z} (1_\bbA - q_n). 
	\]
	Then $p_0=1_\bbA$ and $\mu(p_m -  p_{m+1})=2^{-m-1}$, so $p_m - p_{m+1} \ne 0_\bbA$, for all $m \in \bbN$. Therefore the sets $(p_m- p_{m+1})_{m\in \bbN}$ form a maximal antichain in $\bbA$. 
	
	 Let $G$ be a generic filter, and let $m \in \bbN$ be such that $p_m- p_{m+1}\in G$. We then have $m \in Z$ if and only if $m \notin Z_G$, and therefore $Z \ne Z_G$. Since $Z \subseteq \bbN$ was chosen arbitrarily in $V$, we conclude that $Z_G$ does not belong to $V$.
\end{proof}

If $\bbA$ has an atom $A$, then $G=\{B\in \bbA\mid A\leq B\}$ is a generic filter. On the other hand, \eqref{5.forcing} implies that if $\bbA$ has no atoms then no generic filter exists in $V$. 

\section{Names for elements of metric structures}  \label{S.names}

Completeness of metric spaces is not preserved by forcing extension, since new non-convergent Cauchy sequences might be added in the process. As a consequence, if $M$ is a metric structure in $V$ and $V[G]$ is a generic extension, the domain of $M$ might not be complete in $V[G]$ and $M$ might become a \emph{prestructure} in $V[G]$ in the sense of \cite[\S 3]{benyaacov2008model}. For example, if  a forcing  notion adds a new real number (this is the case with diffuse measure algebras, by Proposition~\ref{P.forcing}.\ref{5.forcing}), then in $V[G]$ even the field of real numbers  that belong to $V$ is no longer complete. 

This issue has an easy solution: take $M^{V[G]}$, which denotes the completion of $M$ in $V[G]$, and interpret it as a metric structure by taking the unique continuous extensions of the interpretations $R^M$, $F^M$ of all symbols in the language of $M$. It is immediate to see that, given $\phi(\bar x)\in \cL_{ \bar x}$, the interpretation $\phi^{M^{V[G]}}$ is the unique continuous function on $(M^{V[G]})^{| \bar x |}$ extending $\phi^M$.

Given a forcing notion $\bbA$, we fix some $\bbA$-name $M^\ast$ for the completion of $M$ in the generic extension and similarly, given $\phi \in \cL_{\bar x}$ (or some symbol in $R, F \in \cL$), we let $\phi^{M^\ast}$ (respectively $R^{M^\ast}, F^{M^\ast}$) be an $\bbA$-name for the function on $M^\ast$ extending $\phi^M$ (respectively $R^M, F^M$).

We denote by $M^\bbA$ the set of all \emph{$\bbA$-names for elements of $M^\ast$}, that is the set of all $\bbA$-names $\dot a$ such that $\fA \dot a \in M^\ast$. This is a proper class. However, if we set $\dot a \sim \dot b$ if and only if $\fA \dot a=\dot b$, then the quotient $\sfrac{M^\bbA}{\sim}$ has cardinality at most $|M|^{\aleph_0}$ if $\bbA$ is a measure algebra. This is a consequence of the fact that measure algebras have countable chain condition  and a standard counting argument (\cite[Lemma~IV.3.5]{kunen2014set}). Alternatively, in each $\sim$-equivalence class we can choose a representative of smallest possible rank and argue that this is a set.

The main result of this section is Lemma~\ref{L.names}, a fact probably well-known to experts, which provides a description of elements in $M^\bbA$, for $\bbA$ the measure algebra of $(\Omega, \mu)$,  in terms of functions in $L^1(\Omega, M)$, with a correspondence analogous to that of names for elements of a Polish space as spaces of functions with possibly infinite values from \cite{jech1985abstract, vaccaro2017generic}. Lemma~\ref{L.names} expands some of the ideas therein to the framework of continuous logic, bringing to light a natural interpretation of $\sfrac{M^\bbA}{\sim}$ as an $\cLR$-structure isomorphic to  the randomization $(\widehat M, \widehat \bbA)$ of $M$. This connection, isolated in Proposition \ref{prop:random},  will be essential for the proof of  Theorem \ref{T.isomorphic_randomizations} in the next section.


\begin{lemma} \label{L.names} Let $(\Omega, \mu)$ be a probability measure space, let $\bbA$ be its measure algebra and let $M$ be a metric structure in some language $\cL$. Then there are functions
\begin{align*}
\Xi \colon L^1(\Omega,M) &\to M^\bbA 
& 		\Theta \colon M^\bbA& \to L^1(\Omega, M)\\
f &\mapsto f(\rAG) 
&  \dot a &\mapsto \Theta_{\dot a}
\end{align*}
such that
	\begin{enumerate}[label=(\alph*)]
	\item \label{2.names}  For every $f \in L^1(\Omega, M)$, $b \in M$, $A \in \bbA$ and $t \in \bbQ_+$
\begin{equation*}
A \fA d^{M^\ast}(f(\rAG), \check b) \le \check t \quad  \iff \quad d^M (f(\omega), b) \le t \quad\text{for $\mu$-almost every } \omega \in A.
\end{equation*}
		\item \label{3.names} Suppose that $\phi(x_0, \dots, x_{k-1}) \in \cL_{\bar x}$, that $f_0, \dots, f_{k-1} \in L^1(\Omega, M)$, that $A \in \bbA$ and $t \in \bbQ_+$. The following are equivalent:
		\begin{enumerate}[label=(\roman*)]
		\item \label{item:i} $A \fA \phi^{M^\ast} (f_0(\rAG), \dots, f_{k-1}(\rAG)) \le \check t$,
		\item \label{item:iii} $ \phi^M(f_0(\omega), \dots , f_{k-1}(\omega)) \le t$ for $\mu$-almost every $\omega \in A$.
		\end{enumerate}	
		\item \label{4.names}  For every $\dot a\in M^{\bbA}$ we have  
		  $\fA \dot a = \Theta_{\dot a}(\rAG)$.
		\item \label{1.names}
		For every $f \in L^1(\Omega, M)$ we have $f(\omega) = \Theta_{f(\rAG)}(\omega)$ for $\mu$-almost every $\omega\in\Omega$.
		\end{enumerate} 
\end{lemma}

\begin{proof}
Fix $f \in L^1(\Omega, M)$. We begin by defining a name $f(\rAG)$ that satisfies condition~\ref{2.names}. If $f = \sum_{i < k} a_i \mathbf{1}_{A_i}$ is simple, then let $f(\rAG)$ be the $\bbA$-name such that $A_i \fA f(\rAG) = \check a_i$ for all $i < n$. It is immediate to check that \ref{2.names} holds in this case.

Suppose next that $f$ is not simple. There is thus a sequence $(f_n)_{n \in \bbN}$ of simple functions that converges pointwise to $f$ for $\mu$-almost every $\omega\in\Omega$. Note that, up to a measure zero set, the range of all $f_n$'s is contained in a separable subset of $M$, whose completion also contains the range of $f$. We can then apply the Egoroff theorem, which, given $\e > 0$, asserts the existence of $B \in \bbA$ such that $\mu (\Omega \setminus B) < \e$ and such that $\sup_{\omega \in B} d^M(f_n(\omega), f(\omega)) \to 0$ as $n \to \infty$, which is to say that $(f_n)_{n \in \bbN}$ converges uniformly to $f$ in $B$. By \ref{2.names} this entails $B \fA ` `(f_n(\rAG))_{n \in \bbN}$ is Cauchy'', and since $\fA ` `M^\ast$ is complete and Hausdorff'', we obtain
\[
B \fA ` `\lim_{n \to \infty}f_n(\rAG) \text{ exists and is unique}''.
\]
Fix now some $A \in \bbA$. By choosing $\e >0$ sufficiently small, we can choose $B$ so that $ B \cap A$ has positive measure. This means that for every $A \in \bbA$ there exists a nonzero $C\subseteq A$ which forces that $\lim_{n \to \infty} f_n(\rAG)$ exists and is unique. By Proposition \ref{P.forcing}.\ref{4.1.forcing}  we deduce $\fA  ` `\lim_{n \to \infty}f_n(r_\bbA)$ exists and is unique''. Using Proposition \ref{P.forcing}.\ref{0.forcing}, we let $f(\rAG) \in M^\bbA$ be an $\bbA$-name for such limit. Note that $f(\rAG)$ does not depend on the sequence $(f_n)_{n \in \bbN}$.

We show that the name $f(\rAG)$ so defined satisfies \ref{2.names}. Fix then $b \in M$, $A \in \bbA$ and $t \in \bbQ_+$. Let $(f_n)_{n \in \bbN}$ be a sequence of simple functions that converges pointwise to $f$ for $\mu$-almost every $\omega\in\Omega$. The equivalence follows by continuity of the distance:
\begin{align*}
d^M(f(\omega), b) \le t &\iff d^M\left(\lim_{n \to \infty}f_n(\omega), b\right) \le t, \quad\text{for $\mu$-almost every } \omega \in A \\
&\iff \lim_{n \to \infty} d^M(f_n(\omega), b) \le t , \quad\text{for $\mu$-almost every } \omega \in A \\
& \iff A \fA \lim_{n \to \infty} d^{M^\ast}(f_n(\rAG), \check b) \le \check t \\
& \iff A \fA d^{M^\ast}\left( \lim_{n \to \infty} f_n(\rAG), \check b\right) \le \check t \\
& \iff A \fA d^{M^\ast}( f(\rAG), \check b)\le \check  t .
\end{align*}

The equivalence in \ref{3.names} is a direct consequence of \ref{2.names} if $f_0, \dots, f_{k-1}$ are simple functions. The general case follows using uniform continuity of $\phi$ and density of simple functions in $L^1(\Omega, M)$, similarly to what we showed above for the distance symbol.

To prove \ref{4.names}, fix $\dot a\in M^\bbA$. Since $\fA ` `M$ is dense in $M^\ast$'', and since $\bbA$ has the c.c.c., for every $n \ge 1$ there is a countable maximal antichain $\cA_n$ and $(a_{B,n})_{B\in \cA_n} \subseteq M$ such that 
\begin{equation} \label{eq:approximate}
B\fA d^{M^\ast}(\dot a, \check a_{B,n}) < \frac{1}{\check n}, \quad  B\in \cA_n.
\end{equation}

Let $\theta_n \in L^1(\Omega, M)$ be defined on $\bigcup_{B\in \cA_n}B$ as:
\[
\theta_n(\omega) = a_{B,n}, \quad   \omega \in B, \, B\in \cA_n,
\]
and define $\theta_n(\omega) = a$ for some $a \in M$, whenever $\omega \in \Omega \setminus \bigcup_{B\in \cA_n}B$. Note that $B \fA \theta_n(\rAG) = \check a_{B,n}$ for all $B \in \cA_n$, so $\fA d^{M^\ast}(\dot a, \theta_n(r_\bbA)) < 1/\check n$ by \eqref{eq:approximate}, and therefore
\begin{equation}\label{eq:limit}
\fA \lim_{n \to \infty}\theta_n(r_\bbA) = \dot a.
\end{equation}

For all $m > n \ge 1$, $B_0\in \cA_m$, $B_1\in \cA_n$, by \eqref{eq:approximate} we have
\begin{equation*}
d^M(\theta_m(\omega),\theta_n(\omega))  <\frac 1m+\frac 1n, \quad\text{for $\mu$-almost every } \omega \in B_0 \cap B_1.
\end{equation*}
This implies in particular that $d^{L^1(\Omega, M)}(\theta_m, \theta_n) \leq \frac 1m+\frac 1n \le \frac2n$, thus $(\theta_n)_{n =1}^\infty$ is a Cauchy sequence in $L^1(\Omega, M)$. Let $\Theta_{\dot a}$ be its limit in $L^1(\Omega, M)$. Then $\fA \Theta(\rAG) = \lim_{n \to \infty}\theta_n(r_\bbA)$, so $\fA \Theta(\rAG) = \dot a$ by \eqref{eq:limit}.

Finally, to prove \ref{1.names}, note that by \ref{4.names} we have $\fA \Theta_{f(\rAG)}(\rAG) = f(\rAG)$, which, by \ref{3.names}, implies $\Theta_{f(\rAG)}(\omega) = f(\omega)$ for $\mu$-almost every $\omega \in \Omega$, as desired.
		\end{proof}

Lemma  \ref{L.names} implies that if $\dot a, \dot b \in M^\bbA$, then $\fA \dot a = \dot b$ if and only if $\Theta_{\dot a} = \Theta_{\dot b}$. From here on, we will systematically  identify elements of $M^\bbA$ with their classes in $\sfrac{M^\bbA}{\sim}$. We moreover retain the notation $\Theta$ for the induced map $\sfrac{M^\bbA}{\sim} \to L^1(\Omega, M)$. Note that, after modding out by $\sim$, $\Theta$ is a bijection with inverse $\Xi$.
By means of this correspondence, we can naturally transpose the $\cLR$-structure defined on $L^1(\Omega, M)$ in \S\ref{S.random} to $\sfrac{M^\bbA}{\sim}$ as follows.
		
\begin{definition} \label{def:random}
Let $M$ be a metric structure over a language $\cL$ and let $\bbA$ be the measure algebra of a probability space $(\Omega, \mu)$. Using $L^1(\Omega, [0,1])$ as an auxiliary sort, and letting $\Theta \colon \sfrac{M^\bbA}{\sim} \to L^1(\Omega, M)$ be the map defined in Lemma \ref{L.names}, we interpret $\sfrac{M^\bbA}{\sim}$ as an $\cL^\mathsf{R}$-structure as follows:
\begin{enumerate}
\item If $\dot a, \dot b \in \sfrac{M^\bbA}{\sim}$, then $d^{\sfrac{M^\bbA}{\sim}}(\dot a, \dot b) \coloneqq d^{L^1(\Omega, M)}(\Theta_{\dot a}, \Theta_{\dot b})$.
\item If $F \in \cL$ is a $k$-ary function symbol and $\dot a_0, \dots, \dot a_{k-1} \in \sfrac{M^\bbA}{\sim}$, we set
\[ 
F^{\sfrac{M^\bbA}{\sim}}(\dot a_0, \dots, \dot a_{k-1}) \coloneqq F^{L^1(\Omega, M)}(\Theta_{\dot a_0}, \dots, \Theta_{\dot a_{k-1}}) (\rAG).
\]
\item If $R$ is a $k$-ary predicate symbol in $\cL$ and $\dot a_0, \dots, \dot a_{k-1} \in M^\bbA$, then we define
\[
\Boo{R(\dot a_0, \dots, \dot a_{k-1})} \coloneqq \Boo{R^{L^1(\Omega, M)}(\Theta_{\dot a_0}, \dots, \Theta_{\dot a_{k-1}})}.
\]
\end{enumerate}
\end{definition}

We note that the interpretation of $\sfrac{M^\bbA}{\sim}$ as an $\cLR$-structure described above can also be recovered in a more direct and natural way, without relying as much on Lemma~\ref{L.names}. Indeed, if $F \in \cL$ is a $k$-ary function symbol and $\dot a_0, \dots, \dot a_{k-1} \in \sfrac{M^\bbA}{\sim}$, then $\fA F^{M^\ast}(\dot a_0, \dots, \dot a_{k-1}) \in M^\ast $ hence, by Proposition \ref{P.forcing}.\ref{0.forcing}, there exists $\dot b\in M^\bbA$ forced to be equal to $F^{M^\ast}(\dot a_0, \dots, \dot a_{k-1})$, and we can take such $\dot b$ to be the interpretation of $F^{\sfrac{M^\bbA}{\sim}}(\dot a_0, \dots, \dot a_{k-1})$. This is equivalent to what is done in Definition~\ref{def:random}, since~$\dot b$ would be in the same $\sim$-class of $F^{L^1(\Omega, M)}(\Theta_{\dot a_0}, \dots, \Theta_{\dot a_{k-1}}) (\rAG)$.

Similarly, if $R \in\cL$ is a  $k$-ary predicate symbol in $\cL$ and $\dot a_0, \dots, \dot a_{k-1} \in M^\bbA$, then $\fA R^{M^\ast}(\dot a_0, \dots, \dot a_{k-1}) \in [0,1]$. We can then interpret $R^{M^\ast}(\dot a_0, \dots, \dot a_{k-1})$ as an $\bbA$-name for a real number which, applying Lemma \ref{L.names} to $M = [0,1]$ as a metric structure in the empty language, naturally corresponds to some element of $L^1(\Omega, [0,1])$. In hindsight, this can be checked to be, up to a measure zero set, precisely the function $\Boo{R^{L^1(\Omega, M)}(\Theta_{\dot a_0}, \dots, \Theta_{\dot a_{k-1}})}$.

The following is an immediate consequence of Definition \ref{def:random}.

\begin{proposition} \label{prop:random}
Let $M$ be a metric structure in a language $\cL$, suppose that $(\Omega, \mu)$ is a probability measure space and let $\bbA$ be its measure algebra. The map $\Theta$ defined in Lemma \ref{L.names}, paired with the identity map on the auxiliary sort  induces an isomorphism of $\cLR$-structures from $(\widehat M, \widehat \bbA)$ onto $\sfrac{M^\bbA}{\sim}$. In particular, for every  $\phi(\bar x) \in \cL_{\bar x}$ and $\dot a_0, \dots, \dot a_{| \bar x | -1} \in \sfrac{M^\bbA}{\sim}$, we have
\begin{equation} \label{eq:equal_random}
\Boo{\phi(\dot a_0, \dots, \dot a_{| \bar x |-1})} = \Boo{\phi^{L^1(\Omega, M)}(\Theta_{\dot a_0}, \dots, \Theta_{\dot a_{| \bar x |-1}})}.
\end{equation}
 \end{proposition}

\section{Probable isomorphisms and isomorphisms of $L^1$-structures}\label{S.X->L1}
In this section we prove Theorem \ref{T.isomorphic_randomizations}, relying on the results of the previous section.

\subsection{From probable isomorphisms to isomorphisms of $L^1$-structures}
The first half of Theorem \ref{T.isomorphic_randomizations} is a consequence of the following result.
\begin{theorem} \label{T.extension_to_random}
	 Let $(\Omega, \mu)$ be a probability measure space, let $\bbA$ be its measure algebra and suppose that $M$ and $N$ are metric structures in the same language $\cL$
		\begin{enumerate}
		\item \label{1.e_t_r} If $\fA  ` ` M^\ast \text{ isometrically embeds into } N^\ast$'', then $L^1(\Omega,M)$ isometrically embeds into $L^1(\Omega ,N)$. 
		\item \label{2.e_t_r} If $\fA  ` ` M^\ast \text{ affinely embeds into } N^\ast$'', then $L^1(\Omega,M)$ affinely embeds into $L^1(\Omega ,N)$. 
		\item \label{3.e_t_r} If $\fA   M^\ast \cong N^\ast$, then $L^1(\Omega,M) \cong L^1(\Omega ,N)$. 
	\end{enumerate} 
\end{theorem}

\begin{proof} Suppose that $\dot \Phi$ is an $\bbA$-name for a function $\dot \Phi\colon M^\ast \to N^\ast$. For every $\dot a \in  M^\bbA$ we have $\fA \exists  y \in N^\ast  \dot\Phi(\dot a)= y$. By Proposition \ref{P.forcing}.\eqref{0.forcing} we can fix an $\bbA$-name $\dot \Phi(\dot a) \in N^\bbA$. We can then define a map $\sfrac{M^\bbA}{\sim} \to \sfrac{N^\bbA}{\sim}$ which, with an abuse of notation, we denote again $\dot \Phi$.

Recall the interpretation of $\sfrac{M^\bbA}{\sim}$, and similarly of $\sfrac{N^\bbA}{\sim}$, as $\cLR$-structures from Definition \ref{def:random}, and \eqref{eq:equal_random} in particular. By Lemma \ref{L.names} we have that, given an $\cL$-formula $\phi( \bar x )$ and $\dot a_0, \dots, \dot a_{k-1} \in \sfrac{M^\bbA}{\sim}$, the following are equivalent
\begin{enumerate}[label=(\roman*)]
\item $\Boo{\phi(\dot a_0, \dots, \dot a_{k-1})} = \Boo{\phi(\dot\Phi(\dot a_0), \dots, \dot\Phi(\dot a_{k-1}))}$.
\item $\fA \phi(\dot a_0, \dots, \dot a_{k-1}) = \phi(\dot\Phi(\dot a_0), \dots, \dot\Phi(\dot a_{k-1}))$.
\end{enumerate}

The second condition holds for all atomic formulas if $\dot \Phi$ is an $\bbA$-name for an isometric embedding, and it holds for all affine formulas if it is an $\bbA$-name for an affine embedding. Items \eqref{1.e_t_r} and \eqref{2.e_t_r} then follow by Proposition \ref{prop:random} and Theorem~\ref{thm:Los}.

Finally, if $\fA   M^\ast \cong N^\ast$, then $\dot \Phi \colon \sfrac{M^\bbA}{\sim} \to \sfrac{N^\bbA}{\sim}$ is also surjective, and so \eqref{3.e_t_r} follows.
\end{proof}

With the previous theorem we can deduce Corollary \ref{C1}. 

\begin{corollary} \label{cor:isomorphic_randomizations}
 Let $(\Omega, \mu)$ be a diffuse probability measure space, let $\bbA$ be its measure algebra. Then $M \coloneqq (\bbR, <)$ and $N \coloneqq (\bbR \setminus \{ 0 \}, <)$ are non-isomorphic linear orders such that $L^1(\Omega, M) \cong L^1(\Omega, N)$.
\end{corollary}

\begin{proof}
		Consider the discrete structures  $M \coloneqq (\bbR,<)$ and  $N \coloneqq (\bbR\setminus \{0\}, <)$.  Since $M$ is Dedekind-complete and $N$ is not, they are not isomorphic. 
	On the other hand, a diffuse measure algebra adds a new real number to the universe by Proposition~\ref{P.forcing}.\ref{5.forcing} and, by \cite[Example 5.1]{laskowski1996forcing}, every forcing $\bbP$ that adds a real also forces $\Vdash_\bbP \check M \cong \check N$. Since $M$ and $N$ are discrete structures we also have $\Vdash_\bbP (M^\ast = \check M) \wedge (N^\ast = \check N)$, hence the conclusion follows by Theorem~\ref{T.extension_to_random}. 
\end{proof}

\subsection{From isomorphisms of $L^1$-structures to probable isomorphisms}\label{S.L1->X}

The following result, combined with Theorem~\ref{T.extension_to_random}, concludes our proof of Theorem \ref{T.isomorphic_randomizations}.
\begin{theorem}\label{T.L1->X}
	 Let $(\Omega, \mu)$ and $(\Xi, \nu)$ be probability measure spaces, and let $\bbA$ and $\bbB$ be the corresponding measure algebras. Suppose that $M$ and $N$ are extremal models of the same non-degenerate, simplicial, affine theory (e.g. single-sorted discrete structures of a complete classical theory). 
	\begin{enumerate}
		\item If $\Phi\colon L^1(\Omega, M)\to L^1(\Xi,N)$ is an affine embedding  then $\Vdash_\bbB$``there is an affine embedding of $M^\ast$ into $N^\ast$''.
		\item If $L^1(\Omega,M)\cong L^1(\Xi,N)$, then $\bbA\cong \bbB$ and $\fA M^\ast \cong N^\ast$. 
	\end{enumerate}
\end{theorem}



\begin{proof}
Let $M$ and $N$ be extremal models of the same non-degenerate, simplicial, affine theory and fix an affine embedding $\Phi \colon L^1(\Omega, M) \to L^1(\Xi, N)$. Suppose that $\frs_\Phi\colon L^1(\Omega,[0,1])\to L^1(\Xi,[0,1])$ is the embedding given by Theorem~\ref{L.DetectingCoordinates}, so in particular, for every formula $\phi(\bar x) \in \Laffx$ and every tuple $\bar f\in L^1(\Omega, M)^{| \bar x |}$ we have 
\begin{equation} \label{eq:preserve}
\frs_{\Phi}\llbracket\phi(\bar f)\rrbracket=\llbracket \phi(\Phi(\bar f))\rrbracket. 
\end{equation}

The image of $\bbA$ via $\frs_\Phi$ is a measure (unital) subalgebra of $\bbB$ with the induced measure, and the map $\frs_\Phi$ naturally extends to $\frs_\Phi \colon V^\bbA \to V^\bbB$ so that, for any formula $\psi(\bar x)$ in  the language of ZFC, $A \in \bbA$ and $\dot \alpha_0, \dots, \dot \alpha_{k-1} \in V^\bbA$ then
\[
A \fA \psi(\dot \alpha_0, \dots, \dot \alpha_{k-1}) \iff \frs_\Phi(A) \Vdash_\bbB \psi(\frs_\Phi(\dot \alpha_0), \dots, \frs_\Phi(\dot \alpha_{k-1})).
\]

After identifying $\bbA$ and $M^\bbA$ with their images in $\bbB$ and $M^\bbB$ by $\frs_\Phi$ respectively, we can use Proposition \ref{prop:random} to see that the map $\Phi$ induces an embedding $\dot \Phi \colon \sfrac{M^\bbA}{\sim} \to \sfrac{N^\bbB}{\sim}$ such that, by \eqref{eq:preserve} and \eqref{eq:equal_random}, for every formula $\phi(\bar x) \in \Laffx$ and $\dot a_0, \dots, \dot a_{k-1} \in \sfrac{M^\bbA}{\sim}$
\[
\llbracket\phi(\dot a_0, \dots, \dot a_{k-1})\rrbracket=\llbracket \phi(\dot \Phi(\dot a_0), \dots, \dot \Phi(\dot a_{k-1}))\rrbracket. 
\]
By Lemma \ref{L.names}, this is equivalent to $\fA  \phi(\dot a_0, \dots, \dot a_{k-1}) = \phi(\dot\Phi(\dot a_0), \dots, \dot\Phi(\dot a_{k-1}))$, which entails that $\dot \Phi$ defines an $\bbB$-name for an affine embedding from a dense subalgebra of $M^\ast$ into $N^\ast$ (if $M$ and $N$ are discrete, then it is already an embedding). This embedding is isometric, and its continuous extension sends $M^\ast$ into $N^\ast$. 

If the map  $\Phi \colon L^1(\Omega, M) \to L^1(\Xi, N)$ is an isomorphism then by Theorem~\ref{L.DetectingCoordinates} $\bbA \cong \bbB$, and the resulting $\bbA$-name $\dot \Phi$ for the continuous extension as in the previous paragraph is moreover forced to have dense range, hence an $\bbA$-name for an isomorphism.
\end{proof}

We record a corollary motivated by a question on von Neumann factors (see Corollary~\ref{C3}). 

The minimal number of Lebesgue null sets needed to cover $[0,1]$ (a cardinal characteristic well-studied and denoted $\cov(\Null)$ in the literature, see e.g., \cite{bartoszynski2009invariants}) is at least $\aleph_1$ (the inequality $\aleph_1\leq \cov(\Null)$ is equivalent to the familiar fact that the Lebesgue measure is countably additive) and equal to the continuum if Martin’s Axiom (MA) holds. MA is consistent with the continuum being greater than $\aleph_\alpha$ for any given ordinal~$\alpha$ (\cite[Theorem V.4.1]{kunen2014set}, also \cite[Theorem~3.4]{Sh:PIF}). 
In the following $\bbL$ is the Lebesgue measure algebra on $[0,1]$ and $M^\ast$ stands for the canonical $\bbL$-name for the completion of $M$.  

\begin{corollary} \label{C2} Suppose that $M$ and $N$ are extremal models of the same non-degenerate, simplicial, affine theory and let $\lambda$ denote the Lebesgue measure on $[0,1]$. 
	\begin{enumerate}
		\item \label{1.C2} Then  $L^1([0,1],\lambda, M)\cong L^1([0,1],\lambda,N)$ if and only if $\Vdash_\bbL M^\ast \cong N^\ast$. 
		\item \label{2.C2} Suppose that $L^1([0,1],\lambda,M)\cong L^1([0,1],\lambda, N)$. Let $\kappa$ be the maximum of density characters of $M$ and $N$ in the strong operator topology. If the real line cannot be covered by $\kappa$ Lebesgue null sets, then  $M\cong N$. 
	\end{enumerate} 
\end{corollary}

\begin{proof} The converse direction of \eqref{1.C2} follows by Theorem~\ref{T.extension_to_random}. 
			The direct implication follows from Theorem~\ref{T.L1->X}.	
	
	We next prove \eqref{2.C2}. 	For convenience, we can present $M$ and $N$ as completions of structures whose domain is $\kappa$. Assume that  $\Phi \colon L^1([0,1],\lambda, M)\to  L^1([0,1],\lambda,N)$ is an isomorphism and 
	 let $\theta$ be a sufficiently large regular cardinal, so that all relevant objects (including $\kappa$,  $M$, $N$, and $\Phi$) belong to $H_\theta$,  the structure of all sets whose hereditary closure has cardinality smaller than $\theta$ (\cite[Definition I.13.27]{kunen2014set} or  \cite[\S A.7]{Fa:STCstar}).   
	 
	Let $P$ be the Mostowski's collapse (\cite[Theorem 6.15]{Jech:ST}) of an elementary submodel of $H_\theta$ of cardinality $\kappa$ that contains $M, N$, and includes $\kappa$ (hence it includes dense subsets of $M$ and $N$ of size $\kappa$), and $\Phi$. The union of all null $G_\delta$ sets coded in~$P$ 	is null by our assumption that the real line cannot be covered by $\kappa$ Lebesgue null sets. Any real $r$ outside of this union is random over~$P$ by Solovay's characterization of random reals (\cite{solovay1970model} or  \cite[Lemma~26.4]{Jech:ST}). 
	Therefore,  $P[r]$ is (isomorphic to) the forcing extension of $P$ by the Lebesgue measure algebra. 
	By item \eqref{1.C2}, $P[r]$ contains an isometric isomorphism $\Psi \colon M^{P[r]} \to N^{P[r]}$. We have that $M^{P[r]}$ is dense in $M$ and $N^{P[r]}$ is dense in $N$  (by the choice of $P$, these structures contain dense subsets of $M$ and $N$ dense in the strong operator topology). It follows that the unique continuous extension of $\Psi$ to $M$ gives an isomorphism between $M$ and $N$. 
	\end{proof}

		\subsection*{Ozawa's `exercise'}\label{Ex.Ozawa}
		N. Ozawa's infamous `exercise' on von Neumann algebra tensor products asks  whether for von Neumann factors $M$ and $N$, $L^\infty([0,1])\bar\otimes M\cong L^\infty([0,1])\bar\otimes N$ implies $M\cong N$ (\cite{ozawa2017exercise}; 
		see \cite{gao2023topics} for some partial answers). The following instance of Corollary~\ref{C2} gives some limitations on what a counterexample may look like. 
		
		\begin{corollary} \label{C3} Suppose that $M$ and $N$ are II$_1$ factors. 
			\begin{enumerate}
				\item \label{1.C3} Then  $L^\infty[0,1]\bar\otimes M\cong L^\infty[0,1]\bar\otimes N$ if and only if $\Vdash_\bbL M^\ast \cong N^\ast$, where $\bbL$ is the Lebesgue measure algebra on $[0,1]$. 
				\item \label{2.C3} Suppose that $L^\infty[0,1]\bar\otimes M\cong L^\infty[0,1]\bar\otimes N$. Let $\kappa$ be the maximum of density characters of $M$ and $N$ in the strong operator topology. If the real line cannot be covered by $\kappa$ Lebesgue null sets, then  $M\cong N$. 
			\end{enumerate} 
		\end{corollary}
		
		\begin{proof}  The theory of tracial von Neumann algebras is non-trivial and simplicial, and its extreme models are precisely the finite factors by \cite[Theorem~29.7]{benyaacov2024extremal}.
			Let~$\tau$ denote the tensor product of the Lebesgue measure $\lambda$ on $L^\infty[0,1]$ and the unique trace on $M$, and consider $L^\infty[0,1]\bar \otimes M$ with the $L^1$-norm $\|a\|_1=\tau(|a|)$. By \cite[Lemma 29.1]{benyaacov2024extremal}, the associated metric is compatible with the strong operator topology on bounded balls, which are thus complete. Therefore $L^\infty[0,1]\bar\otimes M$ can be identified with $L^1([0,1], \lambda,M)$. 
			Therefore 	this is a special case of Corollary~\ref{C2}.		\end{proof}

	\section{Tensorial indecomposability of saturated von Neumann algebras} \label{S.vNa}
	In this section we focus on tensorial indecomposability of ultraproducts of von Neumann algebras and prove Theorem~\ref{T.UltrapowerTensoriallyIndecomposable}. The proof relies on various components, including a key result due to Popa from \cite{popa1983orthogonal} (previously used to prove tensorial indecomposability of ultraproducts of II$_1$ factors in \cite{fang2006central}), as well as some facts relying on forcing from previous sections.
	
	\subsection{Tensorial decomposability in the type I case}

This non-technical section contains the example motivating (and promised after) Definition~\ref{Def.prime}. 

\begin{example}\label{Ex.Prime} Assume the Continuum Hypothesis, CH. Fix  a type II$_1$ tracial von Neumann algebra $(M,\tau)$ with diffuse center. We will prove that if $\cU$ is a nonprincipal ultrafilter on $\bbN$\footnote{This is our only result in which  $\cU$ is required to be a nonprincipal ultrafilter on $\bbN$. In all other results it suffices to have $\cU$ countably incomplete.} then~$M^{\cU}$  absorbs $\ell_\infty(\bbN)$ tensorially. 
	
	Towards this, we claim that, for every $\varepsilon\in (0,1)$, there is a central projection $p$ in~$M$ with trace $\e$ such that $pM^{\cU}$ (equipped with normalized trace $\frac 1{\tau(p)}\tau$) and $M^{\cU}$ are elementarily equivalent. 
	First note the following: if $N$ is a type II$_1$ von Neumann algebra with tracial state $\tau$ and diffuse center, then $N\oplus N$ equipped with the tracial state $(a,b)\mapsto (1-\varepsilon)\tau(a)+\varepsilon \tau(b)$ and~$N$ equipped with $\tau$ are elementarily equivalent. 
	This is because the theories of the direct integral decompositions of $N$ and $N\oplus N$ into II$_1$ factors have the same   distribution of the theories (defined as in \cite[Definition~5.2]{farah2023preservation}) hence  by \cite[Corollary~5.3]{farah2023preservation} they are elementarily equivalent  (see also \cite{gao2024elementary}). Since CH implies $M^\cU$ is saturated, so is $M^\cU\oplus M^\cU$ when equipped with the tracial state $(a,b)\mapsto (1-\varepsilon)\tau(a)+\varepsilon\tau(b)$,  and therefore these two tracial von Neumann algebras are isomorphic for any choice of $\varepsilon>0$. The isomorphism provides a central projection $p$ as required.

	By choosing $\varepsilon_n>0$ so that $\sum_n \varepsilon_n=1$ and repeatedly using the fact from the previous paragraph, we can find a sequence of nonzero central projections~$p_n$ in $M^\cU$ such that $\sum_n p_n=1$ and $p_n M^{\cU}$ is isomorphic~$M^{\cU}$ for all $n$.   Since $\ell_\infty(\bbN)\bar\otimes M^{\cU}$ is isomorphic to $\bigoplus_n M^{\cU}$, the desired conclusion follows. 
\end{example}

%
%
%
%
%
%

	\subsection{Indecomposability by type II$_1$ components}\label{S.Popa}
The following subsection is based on a suggestion by Stefaan Vaes and it is included with his kind permission. The main result is Proposition~\ref{P.CHaar}, stating that countably quantifier-free saturated type II$_1$ von Neumann algebras cannot be decomposed as tensor product of two type II$_1$ von Neumann algebras, and using the key result due to Popa \cite[Corollary~2.6]{popa1983orthogonal}. 

Recall that a unitary $u$ in a II$_1$ factor $(N, \tau)$ is called a \emph{Haar unitary} if $\tau(u^k)=0$ for all $k\geq 1$. 
In the following, we denote by $Z(M)$ the center of a tracial von Neumann algebra $M$, and we let  $E_{Z(M)}$ denote the unique center-valued trace from $M$ onto $Z(M)$ (see e.g. \cite[III.2.5]{blackadar}).

\begin{definition}\label{Def.CHaar}
	Suppose that $M$ is a tracial von Neumann algebra. A unitary $u$ in $M$ is \emph{uniformly Haar} if $E_{Z(M)}(u^k)=0$ for all $k\geq 1$.  
\end{definition}

The terminology \emph{uniformly Haar} is justified by the fact that $E_{Z(M)}(u^k)=0$ for all $k\geq 1$ if and only if $\tau(u^k)=0$  for all tracial states $\tau$ in $M$ and all $k\geq 1$ (see e.g. \cite[\S III.2.5.7]{blackadar}).
In case $M$ is a II$_1$ factor, a unitary is uniformly Haar if and only if it is Haar. More generally, if $M=\int^\oplus_\Omega M_\omega \, d\mu(\omega)$ is a direct integral of II$_1$ factors, then a unitary $u=\int_\Omega u_\omega \, d\mu(\omega)$ in $M$ is uniformly Haar if and only if $u_\omega$ is a Haar unitary in $M_\omega$ for $\mu$-almost every $\omega\in\Omega$. 

\begin{lemma}\label{L.CHaar.0}
	Every type II$_1$ von Neumann algebra contains a uniformly Haar unitary.
	\end{lemma}
\begin{proof}
	Let $\cR$ be the hyperfinite II$_1$ factor and let $u \in \cR$ be a Haar unitary. If $M$ is a type II$_1$ von Neumann algebra, then it contains a unital copy of $\cR$ (see \cite{argerami_overflow} for a detailed proof). We therefore have that $u$  is a uniformly Haar unitary in $M$:  for $k \in \bbN$ we have $E_{Z(M)}(u^k) = 0$ if and only if $\tau(u^k) = 0$ for every trace $\tau$ in $M$, and the conclusion follows since every trace in $M$ restricts to the unique trace $\tau_\cR$ in $\cR$, and $u$ has been chosen so that $\tau_\cR(u^k) = 0$ for every $k \in \bbN$.
		\end{proof}

For countable quantifier-free saturation see \S\ref{S.cqf-saturated}.
\begin{lemma}\label{L.CHaar}
	Suppose that $u$ and $v$ are uniformly Haar unitaries in a countably quantifier-free saturated tracial von Neumann algebra $(M,\tau)$. Then $u$ and $v$ are unitarily equivalent. 
\end{lemma}

\begin{proof}
	It suffices to prove that the type $\bt(x)$ whose (quantifier-free) conditions are $\|x^*x-1\|_2=0$ and $\|xux^\ast -v\|_2=0$ is approximately satisfiable in $M$.
	For this it suffices to prove that there is a sequence of unitaries $(w_n)_{n \in \bbN}$ in $M$ such that $\|w_n u w_n^*-v\|_{2,\tau}\to 0$ as $n\to \infty$. 

Let $M_0$ be an elementary submodel of $M$ with separable predual that contains $u$ and $v$. We thus have $M_0= \int_\Omega^\oplus M_\omega \, d\mu(\omega)$, with $M_\omega$ a II$_1$ factor for $\mu$-almost every $\omega \in \Omega$, and $Z(M_0) \cong L^\infty(\Omega , \mu)$, with $\mu$ being the measure induced by $\tau$ (alternatively, we could also apply disintegration of nonseparable von Neumann algebras as in \cite[\S 8]{benyaacov2024extremal}). In particular, we can assume that $u=\int_\Omega u_\omega\, d\mu(\omega)$ and $v=\int_\Omega v_\omega\, d\mu(\omega)$, where $u_\omega$ and $v_\omega$ are Haar unitaries for $\mu$-almost every $\omega \in \Omega$.

For $n\geq 2$ and $0\leq k<n$ let
\[
J_{n,k}\coloneqq \left\{\exp(i\zeta) \mid \frac{2\pi k} n\leq \zeta<\frac {2\pi(k+1)}n\right\}
\]
	and define $f_n\colon \bbT\to \bbT$ by 
	\[
	f_n(\exp(i\zeta))\coloneqq \exp(2\pi i k/n), \quad \text{if } \zeta\in J_{n,k}. 
	\]
If $\omega \in \Omega$ is such that $M_\omega$ is a factor and both $u_\omega$, $v_\omega$ are Haar unitaries, then $f_n(u_\omega) = \sum_{j < n} \lambda_j p_{j,n}$ and $f_n(v_\omega) = \sum_{j < n} \lambda_j q_{j,n}$, where $\lambda_0, \dots, \lambda_{n-1} \in \bbT$ and
\begin{enumerate}[label=(\roman*)]
\item $\sum_{j < n} p_{j,n} = \sum_{j < n} q_{j,n} = 1_{M_\omega}$,
\item  $\tau_{M_\omega}(p_{j,n}) = \tau_{M_\omega}(q_{j,n}) = \frac1n$, for every $j < n$, where $\tau_{M_\omega}$ is the unique trace on $M_\omega$.
\end{enumerate}
	There exists therefore a unitary $w_{\omega,n} \in M_\omega$ conjugating the tuples $(p_j)_{j < n}$ to $(q_j)_{j < n}$ (see e.g. \cite[Lemma 4.1]{fang2006central}), and as a consequence $w_{\omega,n} f_n(u_\omega) w_{\omega,n}^\ast = f_n(v_\omega)$.
	
	In other words, if $\phi(x,y) \coloneqq \inf_z \| z^\ast z - 1 \|_2 + \| z x z^\ast - y \|_2$, then
	\[
	\phi^{M_\omega}(f_n(u_\omega), f_n(v_\omega)) = 0, \quad
	 \quad\text{for $\mu$-almost every }\omega \in \Omega.
	 \]
	 By \cite[Theorem 8.11]{benyaacov2024extremal} we then get
	\[
	\phi^M(f_n(u),f_n(v)) = \phi^{M_0}(f_n(u),f_n(v)) = \int_\Omega \phi^{M_\omega}(f_n(u_\omega), f_n(v_\omega)))\, d \mu(\omega) = 0.
	\]
	By quantifier-free saturation of $M$ this actually gives a unitary $w_n \in M$ such that $w_n f_n(u) w_n^\ast = f_n(v)$.
	
	 Since $\|f_n(u) -u\|_{2,\tau} \to 0$ and $\|f_n(v)-v\|_{2,\tau}\to 0$ as $n \to \infty$, we therefore get $\|w_n u w_n^*-v\|_2\to 0$ as $n\to \infty$. Finally, by quantifier-free saturation of $M$, $u$ and $v$ are  unitarily conjugate. 
\end{proof}

\begin{proposition}\label{P.CHaar}
	If $(M, \tau)$ is a countably quantifier-free saturated tracial von Neumann algebra of type II$_1$, then $M$ is not isomorphic to the tensor product of two type~II$_1$ von Neumann algebras. 
\end{proposition}

\begin{proof}
	Assume otherwise, so $M\cong P\bar \otimes Q$ with both $P$ and $Q$ of type II$_1$. By Lemma~\ref{L.CHaar.0}, there are uniformly Haar unitaries $u_0 \in P$ and $v_0 \in Q$. Note that both $u \coloneqq u_0 \otimes 1_Q$ and $v \coloneqq 1_P \otimes v_0$ are uniformly Haar in $M$ since $E_{Z(M)}(u) = E_{Z(P)}(u_0) \otimes 1_Q$ and $E_{Z(M)}(v) = 1_P \otimes E_{Z(Q)}(v_0)$.
	
	By Lemma~\ref{L.CHaar}, there is a unitary $w\in M$ such that $wuw^*=v$. In particular, $wuw^*$ is \emph{orthogonal} to $P\bar\otimes 1_Q$, in the sense that $\tau(wuw^\ast a) = 0$ for every $a \in P \bar \otimes 1_Q$.
		By \cite[Corollary~2.6]{popa1983orthogonal}, this implies that $w$ is orthogonal to the normalizer of $P\otimes 1_Q$ in $P\bar \otimes Q$ which, as a von Neumann algebra, generates the whole algebra $ P \bar \otimes Q$, which is a contradiction.
\end{proof}

Popa's \cite[Corollary~2.6]{popa1983orthogonal} has other strong consequences to the structure of countably quantifier-free saturated~II$_1$ factors $M$, usually stated in terms of ultrapowers of II$_1$ factors associated with a nonprincipal ultrafilter on $\bbN$. Such   $M$  cannot be  a crossed product of a diffuse tracial von Neumann algebra by an infinite group, it does not have a Cartan subalgebra (\cite{popa2022topics}), it has no separable maximal abelian subalgebras (\cite{popa2014independence}), and it is not a nontrivial amalgamated free product (see \cite[Proposition 3.16]{gao2024conjugacy}, which also  includes proofs of the previous statements). See also \cite[\S 4.4]{Fa:STCstar} for applications of \cite[Corollary~2.6]{popa1983orthogonal} to \cstar-algebras.

\subsection{Indecomposability}\label{S.Indecomposability}
We are almost ready to prove Theorem \ref{T.UltrapowerTensoriallyIndecomposable}.
	Before doing so, we prove an elementary preliminary lemma. Given a probability measure space $(\Omega, \mu)$, in what follows we use $\tau_\mu$ to denote the trace on $L^\infty(\Omega, \mu)$ induced by~$\mu$.

\begin{lemma}\label{L.commutator.estimate}
Let $(M,\tau)$ be a tracial von Neumann algebra, let $(\Omega, \mu)$ be a probability space. Suppose that $p \in L^\infty(\Omega, \mu) \bar \otimes M$ is a projection and that $w \in L^\infty(\Omega, \mu) \bar \otimes M$ is a unitary. Then $\|[w,p]\|_{2, \tau_\mu \otimes \tau}=1$ if and only if $\|[w(\omega),p(\omega)]\|_{2, \tau}=1$, for $\mu$-almost every $\omega \in \Omega$.  
\end{lemma}

\begin{proof} Only the direct implication requires a proof. 
This follows from the general fact that if $w$ and $p$ are a unitary  and a projection in a tracial von Neumann algebra $(N, \sigma)$, then $\|[w,p]\|_{2,\sigma} \leq 1$. Note first that, since $[w,p]=[w,1-p]$, we may assume $\sigma(p)\leq 1/2$. We moreover have
\begin{align*}
\|[w,p]\|_2^2 &=\tau((wp-pw)^*(wp-pw))\\
& = \tau(2p)-2\tau(pw^*pwp).
\end{align*}
As $pw^*pwp\geq 0$, we deduce that $\|[w,p]\|_2^2 \le 2\tau(p)\leq 1$. 

Now suppose that $w$ and $p$ are a unitary and a projection in $N = L^\infty(\Omega, \mu) \bar \otimes M$. Then $w(\omega)$ and $p(\omega)$ are a unitary and a projection for $\mu$-almost every $\omega\in\Omega$, and, with $f(\omega)=\|[w(\omega),p(\omega)]\|_2^2$, we have
$\|[w,p]\|_{2,\tau_\mu \otimes \tau}^2=\int_\Omega f(\omega)\, d\mu(\omega)$. 
The previous calculation shows that  $0\leq f(\omega)\leq 1$ for $\mu$-almost every $\omega\in\Omega$. 

To complete the proof, it suffices to observe that if $f\colon \Omega\to [0,1]$ is a measurable function such that $\int f\, d\mu=1$ then $f(\omega)=1$ for $\mu$-almost every $\omega\in\Omega$. Assume otherwise. By $\sigma$-additivity of $\mu$ we can fix $n\geq 1$ such that   $\mu(\{\omega: f(\omega)\leq 1-1/n\})\geq 1/m$ for some $m\geq 1$. Then $\int f\, d\mu\leq (1-1/m)+(1-1/n)1/m=1-1/(mn)<1$, contradiction. 
\end{proof}

The following proposition implies Corollary \ref{T.ultrapower} from the introduction, and it is the last ingredient needed to prove Theorem \ref{T.UltrapowerTensoriallyIndecomposable}.
	\begin{proposition}
		\label{L.qf-saturated}
		Suppose that $(P,\tau_P)$ is a tracial type II$_1$ von Neumann algebra and let $(\Omega , \mu)$ be a diffuse probability measure space. Then $M \coloneqq L^\infty(\Omega , \mu) \bar \otimes P$ is not quantifier-free countably saturated.  In particular $M$ cannot be a nontrivial ultraproduct of tracial von Neumann algebras or a relative commutant of a subalgebra with a separable predual in one. 
			\end{proposition}
	
	\begin{proof}

Since $L^\infty(\Omega , \mu)$ is diffuse, we can choose pairwise commuting projections $(q_n)_{n \in \bbN}$ in it such that for all disjoint finite subsets $F$ and $G$ of $\bbN$ we have (with the convention that $\prod_{n\in \emptyset} q_n=\prod_{n\in \emptyset} (1-q_n)=1_Q$)
\[
\textstyle\tau_\mu( \prod_{n\in F} q_n\prod_{n\in G} (1-q_n) )=2^{-|F|-|G|}. 
\]

The following part of the proof is based on the proof that the theory of II$_1$ factors has the order property (\cite[Lemma~3.2]{FaHaSh:Model1}). 
Since $P$ is of type II$_1$, it contains a unital copy of the hyperfinite II$_1$ factor $\cR$. 
Identify $\cR$ with $\overline\bigotimes_{n \in \bbN} P_n$, where $P_n\cong M_2(\bbC)$ and identify each $P_n$ with a unital subalgebra of $P$. 
For every $n \in \bbN$, consider the projections $p_n^0$, $p_n^1 \in P_n$ corresponding to
\[
p_n^0 \coloneqq  \begin{pmatrix} 1& 0 \\ 0 & 0 \end{pmatrix}, \quad p_n^1\coloneqq \begin{pmatrix} 1/2& 1/2 \\ 1/2 & 1/2\end{pmatrix}.
\] 
Finally consider the projections in $L^\infty(\Omega, \mu) \bar \otimes \cR \subseteq M$ defined as 
\[
s^0_n \coloneqq q_n\otimes p_n^0+({1}-q_n)\otimes p_n^1, \quad n \in \bbN, 
\]
\[
s^1_n \coloneqq q_n\otimes p_n^1+({1}-q_n)\otimes p_n^0, \quad n \in \bbN.
\]

Let $\bt(x)$ be the type over $Q\bar\otimes P$ defined as
\[
\bt(x) \coloneqq \{ \|x^*x-1\|_2=0, \, \|[x,s^0_n]\|_2=0, \, \| [x,s^1_n]\|_2=1, \, n \in \bbN \}. 
\]
We show that $\bt$ is consistent in $M$.  Towards this, consider the unitaries in $P_n$
\[
u_n^0\coloneqq\begin{pmatrix} 1 & 0 \\ 0 & -1\end{pmatrix}, \quad 
u_n^1\coloneqq\begin{pmatrix} 0 & 1 \\ 1 & 0\end{pmatrix}, \quad n \in \bbN.
\]
Then $\|[u_n^0, p_n^0]\|_{2,\tau_P}=\|[u_n^1,p_n^1]\|_{2,\tau_P}=0$  and 
$\|[u_n^0, p_n^1]\|_{2,\tau_P}=\|[u_n^1,p_n^0]\|_{2,\tau_P}=1$. 

Since $u_n^i$ and $p_{m}^j$ commute for  all $n\neq m$ and all $i,j = 0,1$, 
the sequence of unitaries
\[
v_n=\prod_{k<n} (q_k \otimes u_k^0+(1-q_k)\otimes u_k^1), \quad n \in\bbN
\]
shows that every finite subset of $\bt(x)$ is satisfiable in $M$.

We deduce that $\bt(x)$ is consistent in $M$. Assume by contradiction that $M$ is countably quantifier-free saturated, and therefore that $\bt(x)$ is realized by some unitary $w \in M$.

We apply now Lemma~\ref{L.names}, and view all $s^i_n$ and $w$  as $\bbA$-names in $P^\bbA$, where $\bbA$ is the measure algebra associated to $(\Omega, \mu)$. After identifying each $q_n$ with the element in $\bbA$ corresponding to its range, by Lemma~\ref{L.names} we get
\begin{equation}
\label{eq.qn}
q_n\fA s_n^0(\rAG)=\check p_n^0,\quad 
\mathbf{1}_\Omega-q_n\fA s_n^1(\rAG)=\check p_n^0, \quad n \in \bbN.
\end{equation}
By  Lemma \ref{L.commutator.estimate} we moreover obtain
\begin{equation*}
\| [w(\omega), s_n^0(\omega) ] \|_{2, \tau_P} = 0, \quad \| [w(\omega), s_n^1(\omega) ] \|_{2, \tau_P} = 1,
\quad\text{for $\mu$-almost every }\omega \in \Omega, \, n \in \bbN.
\end{equation*}
Therefore, Lemma~\ref{L.names}, along with \eqref{eq.qn}, implies
\begin{equation*}
q_n \fA \|[w(\rAG), \check p^0_n] \|_{2, \tau_{P^\ast}} = 0,  \quad\mathbf{1}_\Omega - q_n \fA \|[w(\rAG), \check p^0_n] \|_{2, \tau_{P^\ast}} = 1, \quad n \in \bbN.
\end{equation*}

As a consequence, in the generic extension $V[G]$ we get that the evaluation $w(\dot r_\bbA)_G$ belongs to $P^{V[G]}$, and, more importantly,
\[
[w(\dot r_\bbA)_G, p_n^0] = 0 \text{ if and only if }q_n \in G.
\]
Since $P$ is dense in $P^{V[G]}$, let $w_0 \in P$ be such that $\| w_0 - w(\dot r_\bbA)_G \|_{2, \tau_P^{V[G]}} < 1/4$. We therefore get
\[
Z(G) \coloneqq \{ n \in \bbN : q_n \in G \} = \{ n \in \bbN : \| [w_0, p_n^0] \|_{2, \tau_P} < 1/2 \}.
\]
Since the set on the right-hand side belongs to $V$, we obtain a contradiction by Proposition \ref{P.forcing}.\ref{5.forcing}.
	\end{proof}

	\begin{theorem}
		\label{T.CountablySaturated} Suppose that $M$ is a diffuse, countably quantifier-free  saturated, tracial von Neumann algebra. Then $M$ is tensorially prime if and only if it is not of type I. 
	\end{theorem}

	\begin{proof} We first prove that every countably quantifier-free saturated tracial von Neumann algebra $M$ that is not type I is tensorially prime.
	Assume otherwise, so $M=Q\bar \otimes P$ where each one of $P$ and $Q$ has a unit ball that is not compact in the 2-norm. Since $M$ is not of type I, at least one of $P$ or  $Q$ is not of type I; we may assume that $P$ is not of type I. Since the corner of a countably quantifier-free saturated tracial von Neumann algebra is countably quantifier-free saturated,   by replacing $P$ with its corner, we may assume it is of type II$_1$. 
	Since the unit ball of $Q$ is not compact in the 2-norm, there is  a nonzero projection $q_0\in Q$ such that the corner $q_0Qq_0$ is diffuse. By type decomposition, there is a nonzero projection $q_1\leq q_0$ such that $q_1 Q q_1$ is of type I or of type II$_1$.  We may therefore assume that $Q$ is diffuse and of type I or of type II$_1$. 
	In the former case, Proposition~\ref{L.qf-saturated} implies that $M$ is not countably quantifier-free saturated. 
	In the latter case, Proposition~\ref{P.CHaar} implies that $M$ is not countably quantifier-free saturated. Since both cases lead to contradiction, every countably quantifier-free saturated tracial von Neumann algebra that is not type I is tensorially prime.

			For the second part of the theorem, fix a type I countably quantifier-free saturated tracial von Neumann algebra $M$. By a combination of characterization of type I tracial von Neumann algebras and Proposition~\ref{P.Abelian-tvNa}, we have that $M$ is isomorphic to a direct sum of von Neumann algebras of the form $M_n(A_\kappa)$,  where $n\geq 1$ and $A_\kappa=L^\infty(\{0,1\}^\kappa,\mu_\kappa)$. 
			Since $M$ is countably quantifier-free saturated, so is each of its summands. As $M$ is diffuse, so is each of its summands. Being countably quantifier-free saturated and diffuse, these summands are  necessarily nonseparable.
			For infinite cardinals $\kappa$ and $\lambda$ we have $A_\kappa\bar\otimes A_\lambda\cong A_{\max(\kappa,\lambda)}$. Hence each of the direct summands  of $M$ absorbs $L^\infty[0,1]$ tensorially, and therefore so does $M$.
	\end{proof}
	
		\begin{proof}[Proof of Theorem~\ref{T.UltrapowerTensoriallyIndecomposable}]  	
			Suppose that $M$ is an ultraproduct  of tracial von Neumann algebras associated with a countably incomplete ultrafilter on $\bbN$. Then $M$ is countably quantifier-free saturated, and so is $M\cap N'$ for every \wstar-subalgebra $N$ of $M$ with separable predual. 
Therefore the conclusion follows by Theorem~\ref{T.CountablySaturated}. 
				\end{proof}

		\section{Concluding remarks}\label{S.Concluding}
		
		We ought to mention that this is not the first direct application of forcing to operator algebras. In \cite{ozawa1985nonuniqueness}, M. Ozawa  proved that for every pair of infinite cardinals $\kappa_0<\kappa_1$ there is an \awstar-module $M$ over a commutative \awstar-algebra that has one orthonormal basis of cardinality $\kappa_0$ and one of cardinality $\kappa_1$.  Such \awstar-module is both $\kappa_0$-homogeneous and $\kappa_1$-homogeneous. This confirmed a conjecture of Kaplansky.  
		  The proof uses the L\' evy poset $\bbP$ for collapsing $\lambda$ to $\kappa$ with conditions of cardinality $<\kappa$. Taking $\cB$ to be the complete Boolean algebra of regular open subsets of $\bbP$, the required \awstar{} algebra is obtained directly by Stone and Gelfand--Naimark dualities: it is $C(X)$, where $X$ is the Stone space of $\cB$.

		\subsection*{Probably isomorphic structures}
It is about time that we defined the notion that this paper is about. 

\begin{definition}\label{D.Probably}
		If the Lebesgue measure algebra forces the completions of two structures $M$ and $N$ of the same language to be isomorphic,\footnote{Since the Lebesgue measure algebra is homogeneous as a forcing notion, this does not depend on the choice of the generic filter.} we say that $M$ and $N$ are \emph{probably isomorphic}. 
\end{definition}
 Theorem~\ref{T.isomorphic_randomizations} implies that if $M$ and $N$ are discrete structures, or  extremal models of the same non-degenerate, simplicial theory, then they are probably isomorphic if and only if $L^1([0,1],\lambda,  M) \cong L^1([0,1],\lambda, N)$.
		
		\begin{question}\label{Q1}
			Is there a characterization of theories that admit probably isomorphic, but not isomorphic models? 
		\end{question}
		
		One such theory is the theory of dense linear orders, because by  \cite[Example~5.1]{laskowski1996forcing}, linear orderings $\bbR$ and $\bbR\setminus \{0\}$ are probably isomorphic but not isomorphic. This example was used in the proof of Corollary~\ref{C1}.  By Theorem~\ref{T.isomorphic_randomizations},  Ozawa's `exercise' \cite{ozawa2017exercise} literally asks whether the theory of II$_1$ factors has this property. 
		The study of nonisomorphic models such that a forcing with the countable chain condition makes them isomorphic was initiated in \cite{baldwin1993forcing} and \cite{laskowski1996forcing}  and more recently the analogous question for $\sigma$-closed forcings was studied in \cite{shelah2025twins}. 
		
		Questions analogous to Question~\ref{Q1} can be asked for Cohen forcing and other `nice' forcing notions, although we are not aware of natural analogs of Theorem~\ref{T.isomorphic_randomizations} in other contexts.

		\subsection*{Probably isomorphic ICC groups}
		 
		A discrete group $\Gamma$ has  ICC (infinite conjugacy classes) if every conjugacy class except that of the identity is infinite. A group~$\Gamma$ is ICC if and only if the group von Neumann algebra $L(\Gamma)$ is a II$_1$ factor (see \cite[Proposition 1.3.9]{anantharaman2017introduction}). 
		
		\begin{question} Are there nonisomorphic ICC groups $\Gamma_0$  and $\Gamma_1$ such that the Lebesgue measure algebra forces $\Gamma_0\cong \Gamma_1$?
		\end{question}
		
		An affirmative answer to this question would not outright imply that there are nonisomorphic II$_1$ factors that can be forced to be isomorphic, as it is well-known that group factors associated to  non-isomorphic ICC groups can be isomorphic. For example, each amenable ICC group gives rise to a hyperfinite II$_1$ factor, and there is a unique hyperfinite II$_1$ factor with separable predual (i.e., whose unit ball is separable in the $\|\cdot\|_2$-norm). The question whether all von Neumann factors associated with nonabelian (countable) free groups are isomorphic, is one of the most prominent questions in the theory of von Neumann algebras. Since $L(F_n)$ is isomorphic to the free product of $n$ copies of $L^\infty[0,1]$,  this leads to the question whether for a  given abelian tracial von Neumann algebra $A$ and $m\neq n$ a free product of $m$ copies of~$A$
		can be isomorphic to the free product of $n$ copies of~$A$. As the referee pointed out,  there are many questions of similar flavour, but we decide to stop here.  The case when $A$ is nonseparable has a negative answer by \cite{boutonnet2024non}. 
		
		\subsection*{Continuous ultrapowers}
		The generic section of $A\bar\otimes P$ for an abelian tracial von Neumann algebra $A$ and II$_1$ factor $P$, used implicitly in the proof of Lemma~\ref{L.qf-saturated}, is equal to the completion $P^{V[G]}$ of $P$ in the forcing extension.   
 This construction appears to be related to the continuous ultrapower $M^{/E,\cU}$ as in \cite[Definition~2.1]{gao2024elementary}.\footnote{In this notation $M$ is a direct  integral of tracial von Neumann algebras, $E$ is the conditional expectation onto its center $Z(M)$, and $\cU$ is a character on $Z(M)$.} However, while the Gao--Jekel continuous ultrapower is countably saturated (\cite[Proposition~5.7]{gao2024elementary}), the generic section of $L^\infty[0,1]\bar\otimes P$ is isomorphic to $P$ and therefore not countably saturated unless $P$ is.   

\subsection*{Non-structure}
		Shelah’s non-structure theory (\cite{shelah2023general}) was used in  \cite{FaKa:NonseparableII} to define a functor from the category of linear orderings into II$_1$ factors such that sufficiently different linear orderings are sent to nonisomorphic hyperfinite II$_1$ factors. This was used  to prove that in every uncountable density character $\kappa$ there are $2^\kappa$ nonisomorphic hyperfinite II$_1$ factors. 
		However, the nonisomorphic linear orderings such that the Lebesgue measure algebra forces that they are isomorphic from \cite{laskowski1996forcing} and used in Corollary~\ref{C1} are not sufficiently different. More precisely,   being separable, they have the same isomorphism invariant that is usually used to distinguish the corresponding Ehrenfeucht--Mostowski models (see e.g., \cite{shelah2023general} or \cite{FaSh:Dichotomy}). 
		Moreover, \emph{sufficiently different} is defined using an invariant that codes the gap structure of a given linear ordering, and this structure cannot be changed by any Lebesgue measure algebra, or any ccc forcing. It is however not obvious  whether the II$_1$ factors associated with $\bbR$ and $\bbR\setminus \{0\}$ are isomorphic.

	A first-order  theory $T$ has the \emph{independence property} (a common expression is that $T$ does not have  NIP, where NIP stands for `no independence property') if there is a formula $\varphi(\bar x,\bar y)$ such that in every model $M$ of $T$, for every $n$ there are $\bar a_i$, for $i<n$ such that for all $s\subseteq n$ there is $\bar b_s$ such that $M\models \varphi(\bar a_i,\bar b_s)$ if and only if $i\in s$ (\cite{shelah1990classification}).

	The proof of Proposition~\ref{L.qf-saturated} provides some evidence for the following.

	\begin{conjecture}\label{C.IP} If $M$ has the independence property witnessed by an affine quantifier-free formula, $(\Omega, \mu)$ is a diffuse probability space, and $\cU$ is a countably incomplete ultrafilter, $L^1(\Omega,M)^{\cU}$ is not isomorphic to $L^1(\Lambda,N)$ for any diffuse measure space $\Lambda$ and any structure $N$. 
	\end{conjecture}	
	
	\appendix
	
		\section{Ioana's proof of Theorem~\ref{T.UltrapowerTensoriallyIndecomposable}} \label{S.def/rg}

	Proposition~\ref{P.def/rg} subsumes Proposition~\ref{P.CHaar} and Proposition~\ref{L.qf-saturated}. Its proof, due to Adrian Ioana,  uses only standard methods of Popa's deformation/rigidity theory. Together with the first paragraph of the proof of Theorem~\ref{T.CountablySaturated} it gives a proof of Theorem~\ref{T.UltrapowerTensoriallyIndecomposable}.   	
	
		\begin{proposition}\label{P.def/rg}
Suppose that $(P,\tau_P)$ and $(Q, \tau_Q)$ are diffuse tracial von Neumann algebras, and assume that $P$ is not of type I. Then the tensor product von Neumann algebra $M=P\bar\otimes Q $ is not countably quantifier-free saturated. 		
\end{proposition}
		
\begin{proof}	
		Assume otherwise. By Lemma~\ref{L.CHaar.0} choose a uniformly Haar unitary  $w_0 \in P$  and a Haar unitary  $w_1\in Q$. We claim that both $u_0\coloneqq w_0\otimes 1_Q$  and $u_1\coloneqq w_0 \otimes w_1$ are uniformly Haar unitaries in $M$. For $u_0$, this is shown in the proof of Proposition~\ref{P.CHaar}. For $u_1$, this follows by observing that, since $E_{Z(P)}(w_0^k)=0$, we have    $E_{Z(M)}(u_1^k)=E_{Z(P)}(w_0^k)\otimes E_{Z(Q)}(w_1^k)=0$ for all $k\geq 1$. 
		
		By countable quantifier-free saturation, Lemma~\ref{L.CHaar}, implies that $u_0$ and $u_1$ are unitarily equivalent. Fix a unitary $v\in M$ such that $vu_1 v^*=u_0$. Then $v(w_0\otimes w_1)v^*=w_0\otimes 1_Q$, hence $v(w_0^k\otimes w_1^k)v^*=w_0^k\otimes 1_Q$ for every $k\geq 1$. This implies that $\|E_{P\otimes 1_Q}(v(w_0^k\otimes w_1^k)v^*)\|_{2, \tau_P}=1$ for every $k\geq 1$, where $E_{P \otimes 1_Q} \colon M \to P$ is the unique conditional expectation mapping $a \otimes b$ to $\tau_Q( b) a$.
		
		\begin{claim} \label{claim} For all $ a,b \in M$ we have  $\lim_{k\to \infty}\|E_{P\otimes 1}(a(w_0^k\otimes w_1^k)b)\|_{2,\tau_P}=0$.
	\end{claim}
	\begin{proof}
		It suffices to prove this for a $\|\cdot\|_{2, \tau_P \otimes \tau_Q}$-dense set of pairs $a,b$ of elements in~$M$, and by linearity it suffices to prove the claim for elementary tensors.  
We may therefore assume that $a=a_0\otimes a_1$ and $b=b_0\otimes b_1$ for some $a_0, b_0 \in P$ and $a_1, b_1 \in Q$. Then, for every $k \in \bbN$, we have $a(w_0^k\otimes w_1^k)b=(a_0w_0^kb_0)\otimes (a_1w_1^kb_1)$, hence 
		\[
		E_{P\otimes 1}(a(w_0^k\otimes w_1^k)b)=\tau_Q(a_1w_1^kb_1) a_0w_0^kb_0.
		\] 
		Since~$w_1$ is a Haar unitary, by a standard and easily checked fact $w_1^k$ converges to $0$ weakly, therefore $\lim_{k\to \infty}\tau_Q(a_1w_1^kb_1)=\langle w_1^k,b_1 a_1^\ast\rangle_{\tau_Q}=0$. Together with the previous calculation this implies 
		\[
		\lim_{k\to \infty}\|E_{P\otimes 1}(a(w_0^k\otimes w_1^k)b)\|_{2,\tau_P}=0.
		\]
		\end{proof}
		
		 The instance of the claim when $a=v$ and $b=v^*$ gives contradiction and completes the proof.  
		 \end{proof} 
		 
		 As Adrian Ioana also pointed out, Claim \ref{claim} has another interesting consequence stated in terms of Popa's deformation and rigidity theory (see e.g., \cite[\S 2]{popa2006strong}).  
		 Using the terminology and notation from the proof of Proposition~\ref{P.def/rg},  Claim \ref{claim} shows that $\wst(w_0\otimes w_1)$ does not intertwine into $P\otimes 1_Q$, and therefore not into $\wst(w_0\otimes 1_Q)$, inside $M$.

\bibliographystyle{amsalpha}
\bibliography{Bibliography}

\end{document}